\documentclass[reqno]{amsart}
\usepackage{amsmath,amsfonts,amssymb,graphics,graphicx,latexsym,amscd,subfigure,hyperref}
\usepackage[usenames]{xcolor}
\usepackage{url}
\catcode`~=11 \def\UrlSpecials{\do\~{\kern -.15em\lower .7ex\hbox{~}\kern .04em}} \catcode`~=13

\DeclareMathAlphabet{\mathantt}{OT1}{antt}{li}{it}
\DeclareMathAlphabet{\mathpzc}{OT1}{pzc}{m}{it}

\newtheorem{thm}{Theorem}[section]

\newtheorem{prop}[thm]{Proposition}

\newtheorem{lemma}[thm]{Lemma}
\newtheorem{defn}[thm]{Definition}

\newcommand{\bbQ}{\mathbb{Q}}
\newcommand{\bbR}{\mathbb{R}}
\newcommand{\bbT}{\mathbb{T}}
\newcommand{\bbC}{\mathbb{C}}
\newcommand{\bbZ}{\mathbb{Z}}

\newcommand{\s}{g}
\newcommand{\n}{\eta}
\newcommand{\p}{\mathfrak{p}}
\newcommand{\pr}{\text{pr}}

\newcommand{\Sm}{\mathcal{C^\infty}}
\newcommand{\hess}{\operatorname{Hess}}
\newcommand{\interior}{\operatorname{int}}
\newcommand{\Lap}{\triangle}
\newcommand{\op}{\mathcal{P}}

\newcommand{\tr}{\operatorname{Tr}}

\newcommand{\vol}{\text{vol}}

\newcommand{\action}{\tau}

\newcommand{\num}{\mathcal{N}}
\newcommand{\param}{\nu}

\def\det{\mathop{\rm det}\nolimits}

\def\dim{\mathop{\rm dim}\nolimits}
 
\def\sgn{\mathop{\rm sgn}\nolimits}
\def\spec{\mathop{\rm spec}\nolimits}

\def\tr{\mathop{\rm tr}\nolimits}

\def\vol{\mathop{\rm vol}\nolimits}

\newcommand{\fp}{{\mathfrak p}}

\newcommand{\RR}{{\mathbb R}}

\begin{document}

\title[]{Semi-classical weights and equivariant spectral theory}

\author{Emily B. Dryden}
\address{Department of Mathematics, Bucknell University, Lewisburg, PA 17837, USA}
\email{ed012@bucknell.edu}
\author{Victor Guillemin}
\address{Department of Mathematics, Massachusetts Institute of Technology, Cambridge, MA 02139, USA}
\email{vwg@math.mit.edu}
\author{Rosa Sena-Dias}
\address{Departamento de Matem\'{a}tica, Instituto Superior T\'{e}cnico, Av. Rovisco Pais, 1049-001 Lisboa, Portugal}
\email{rsenadias@math.ist.utl.pt}

\date{}

\begin{abstract}

We prove inverse spectral results for differential operators on manifolds {and orbifolds} invariant under a torus action. These inverse spectral results involve the asymptotic equivariant spectrum, which is the spectrum itself  together with ``very large" weights of the torus action on eigenspaces. More precisely, we show that the asymptotic equivariant spectrum of the Laplace operator of any toric metric on a generic toric orbifold determines the equivariant biholomorphism class of the orbifold; we also show that the asymptotic equivariant spectrum of a $\bbT^n$-invariant Schr\"odinger operator on $\bbR^n$ determines its potential in some suitably convex cases. In addition, we prove that the asymptotic equivariant spectrum of an $S^1$-invariant metric on $S^2$ determines the metric itself in many cases. Finally, we obtain an asymptotic equivariant inverse spectral result for weighted projective spaces. As a crucial ingredient {in these inverse results,} we derive a surprisingly simple formula for the asymptotic equivariant trace of a family of semi-classical differential operators invariant under a torus action.

\end{abstract}

\maketitle


\section{Introduction}\label{sec:intro}

To what extent are the geometric properties of a Riemannian manifold determined by the spectrum of its Laplacian? In \cite{dgs2} the authors study a variation of this classical problem in the context of an isometric group action. This gives rise to the notion of  {\it equivariant spectrum}, which is simply the spectrum together with the weights of the representations of the isometric group action on the eigenspaces. A classical tool used to answer inverse spectral questions is an asymptotic expansion for the heat kernel or the wave trace for the Laplacian. With this in mind, it is natural to try to use such an expansion for the equivariant heat kernel or the equivariant wave trace, thereby ``counting" equivariant eigenfunctions. 

In \cite{dgs3} the authors obtain asymptotic expansions for traces of operators on manifolds in the presence of an isometry. The goal of the present paper is to develop the techniques in \cite{dgs3} to obtain an asymptotic expansion for the equivariant trace of a semi-classical differential operator invariant under a group action; the equivariant trace encodes information about equivariant eigenspaces. One of our key observations is that by considering ``very large" or {\it semi-classical} weights, {which give rise to the asymptotic equivariant spectrum,} it is possible to obtain a relatively simple formula for the asymptotic behavior of this equivariant trace. We now give a rough formulation of this result; a more precise version can be found in Theorem \ref{mainthm}.  

Consider a Riemannian manifold $X^n$ admitting an isometric action of a {torus $\bbT^m$}. This action lifts to a Hamiltonian action on $T^*X$ with moment map $\Phi: T^*X \rightarrow {\bbR^m}$. For generic $\alpha \in {\bbR^m}$, let $(T^*X)_\alpha:=\Phi^{-1}(\alpha)/{\bbT^m}$ denote the symplectic reduction of $T^*X$ at level $\alpha$, and let $\chi_\alpha$ be a character of ${\bbT^m}$ associated with $\alpha$. Let $\op_h$ be a family of semi-classical differential operators that are ${\bbT^m}$-invariant. In applications we will consider $\op_h=h^2\Lap$ or $\op_h=h^2\Lap+V$. We denote by $\fp_0:T^*X\rightarrow \bbR$ the leading symbol of $\op_h$. This induces a map $\p_\alpha:(T^*X)_\alpha\rightarrow \bbR$. Let $\lambda_i(\alpha,h)$ for $i=1,2,\dots$ denote the $\frac{\alpha}{h}$-equivariant eigenvalues counted with multiplicities, that is, the eigenvalues of $\op_h$ restricted to $C^\infty(X)^{\frac{\alpha}{h}}$, where $C^\infty(X)^{\frac{\alpha}{h}}$is the space of functions
$$
\Sm(X)^{\frac{\alpha}{h}}=\{f\in \Sm(X): f(gx)=\chi_{\frac{\alpha}{h}}(g)f(x)\}.
$$
Consider the spectral measure
\begin{equation}\label{spectral_measure}
\mu_{\frac{\alpha}{h}}(\rho) := \sum_i \rho(\lambda_i(\alpha,h)),
\end{equation}
where $\rho \in C_0^\infty(\bbR)$ {and the eigenvalues are repeated according to their multiplicities.} 
This is well defined provided $\frac{1}{h}$ is an integer{, and} admits an asymptotic expansion in powers of $h$. 

\begin{thm}
With the setup given above, let $\frac{\alpha}{h}$ be a weight of the $\bbT^m$-action {$\tau$} such that $\alpha$ is a regular value for the moment map $\Phi:T^*X\rightarrow \bbR^m$. Consider all the isotropy groups for the $\bbT^m$ action which are finite and label the non-identity elements in them as $ e^{i\theta_1}, \dots, e^{i\theta_N}$. Then
{
\begin{align}\label{asymp_expansion}
\mu_{\frac{\alpha}{h}}(\rho) &\sim& (2\pi h)^{m-n}\sum_{i=0}^\infty h^i \sum_{\ell \leq 2i}\mu_{i,\ell} \left( \frac{d^{\ell}\rho}{dt^{\ell}} \right) \hspace*{5cm} \notag \\ 
&&+ \  
\sum_{r,s} (2 \pi h)^{{-k}} \chi_{\frac{\alpha}{h}}(e^{-i\theta_r}) |\operatorname{det}(A_r-I)|^{-1} \left( \mu_{r,s} +O(h^{-{k}+1})\right),
\end{align}
where $\mu_{i, \ell}$ and $\mu_{r,s}$ are measures; $s$ is a positive integer depending on $r=1, \dots, N$; $k \geq m$ depends on $r$ and $s$; and $A_r$ is the linear action of $\tau(e^{i\theta_r})$ restricted to a certain normal space. Furthermore we have that the leading term {of the expression in the first line} is $(\fp_\alpha)_*\nu$}, where $\fp_\alpha$ is the reduced symbol of $\op_h$ on the reduced cotangent space $(T^*X)_\alpha$, and $\nu$ is the volume induced by the reduced symplectic form $\omega_\alpha$.  
\end{thm} 
We will also be able to describe the {form of the terms, and especially the leading term, arising from the second line of the expansion}. A more precise version of this theorem is Theorem \ref{mainthm}.
Although this formula is interesting in its own right, we will show that it leads to several {asymptotic} equivariant inverse spectral results.  We emphasize that these results involve the asymptotic equivariant spectrum, i.e., the spectrum together with the ``large" weights of the group action on the eigenspaces. {For instance}, for so-called \emph{admissible} potentials associated to symmetric Schr\"{o}dinger operators, we prove in \S \ref{sec:inverse_spectral_results} that the potential is determined by the asymptotic equivariant spectrum of the operator.

\begin{thm}
Let $V{(s_1, \dots, s_n)}$ be an admissible function on {$\bbR^{n}_+$}.
Consider the Schr\"odinger operator given by $-h^2\Lap+V$, where $\Lap = - \sum \left(\frac{\partial^2}{\partial x_i^2} + \frac{\partial^2}{\partial y_i^2}\right)$ and $V = V(x_1^2+y_1^2, \dots, x_n^2+y_n^2)$. Then $V$ is determined by the asymptotic equivariant spectrum of ${-h^2}\Lap+V$.
\end{thm}
In the case of toric {orbifolds}, we prove that the asymptotic equivariant spectrum of the Laplacian of any {toric K\"ahler} metric determines the {(unlabeled)} moment polytope of a generic toric orbifold.  In particular, this implies that the asymptotic equivariant spectrum determines the orbifold up to equivariant biholomorphism.
\begin{thm}
Let $(X,\omega)$ be a generic toric orbifold endowed with any toric K\"ahler metric. The moment polytope of $X$ is determined by the asymptotic equivariant spectrum of the Laplacian{, up to two choices and up to translation.}
\end{thm}
In the case of $S^2$, we go further to obtain some partial inverse spectral results on the metric. 
\begin{thm}
Given any $S^1$-invariant metric on $S^2$ with symplectic potential $\s$, if $\ddot \s$ is even and convex, then $\s$ is determined by the asymptotic equivariant spectrum of the Laplacian of the metric.
\end{thm}

Our analytic result is of a similar flavor to results of Br\"{u}ning and Heintze \cite{bh1,bh2} and Donnelly \cite{d}, although we are concerned with different spectral measures.  These authors {consider} the Laplace operator restricted to the space of $G$-invariant functions for general compact Lie groups $G$, and determine the leading term of the asymptotic expansion of the heat kernel corresponding to {the heat kernel of the Laplacian on this space}.  In particular, they prove that {the eigenvalues of this heat kernel} satisfy a Weyl law similar to the usual Weyl law for the Laplacian on the full space of functions on $X$.  Moreover, Br\"{u}ning and Heintze \cite{bh2} obtain a full asymptotic expansion of the heat trace of the equivariant Laplacian, giving the multiplicity with which a given representation occurs as a representation on the eigenspaces of the Laplacian.  The expansion becomes complicated if there are many different types of $G$ orbits; in particular, it contains logarithmic terms.  One problem encountered by Br\"{u}ning and Heintze and Donnelly is that the space $X/G$ generally is not a manifold, so accommodations must be made if one {wants} to think of $\Lap_{|\Sm(X)^0}$ as a Laplace operator on $X/G$.  The same difficulties appear for us, but unlike these other authors, we address them by using semi-classical weights.  We note that the subset of $X$ where the action of $G$ is locally free, which we denote by $X_0$, plays a fundamental role in these various approaches.

The main point for us is that $\alpha$ is a regular value of $\Phi$ if and only if $\Phi^{-1}(\alpha)$ is contained in $T^*X_0$. The fact that the operators $\op_h$ are invariant under the ${\bbT^m}$ action implies that their symbol carries over to $(T^*X)_\alpha$, and we show that we can identify $(T^*X)_\alpha$ with $T^*(X_0/{\bbT^m})$.  This seems to leave us with the same compactification problems that were faced in \cite{bh2}, but we will briefly explain why this is not the case.  The leading symbol $\p_0:T^*X\rightarrow \bbR$ of $h^2\Lap$ on $X_0$ is given by $\p_0(x,\xi)=||\xi||^2_x$, where $||\xi||^2_x$ denotes the Riemannian metric on $T_x^*X$.  We have an orthogonal splitting
$$
T^*X_0=\pi^*T^*(X_0/{\bbT^m})\oplus {\bbR^m}
$$
in terms of which
$$
\Phi^{-1}(\alpha)=\pi^*T^*(X_0/{\bbT^m})+\alpha.
$$
The restriction of $\p_0$ to $\Phi^{-1}(\alpha)$ is therefore given by
$\p_0(x,\eta)+||\alpha||^2_x$, where we write $\xi=\eta+\alpha$ with $\eta\in \pi^*T^*(X_0/{\bbT^m})$ because $\xi \in \Phi^{-1}(\alpha)$.  In other words, the restriction to $\Phi^{-1}(\alpha)$ of the leading symbol of $h^2\Lap$ becomes the function $\p_\alpha+V_\alpha$, 
where $\p_\alpha$ is the restriction to $\pi^*T^*(X_0/{\bbT^m})$ of the symbol $\p_0$ and $V_\alpha$ is a potential function given by $V_\alpha=||\alpha||^2_x$.  We can also identify $(T^*X)_\alpha=\Phi^{-1}(\alpha)/{\bbT^m}$ with $\pi^*T^*(X_0/{\bbT^m})/{\bbT^m}=T^*(X_0/{\bbT^m})$ and, via this identification, the reduced symbol of $h^2\Lap$ becomes the function 
$\pi^*\fp_\alpha=\fp_0+V_\alpha$.  The upshot of our approach is that $V_\alpha$ tends to infinity at the boundary points in $X \setminus X_0$.  Thus the function $\p_\alpha:T^*(X_0/{\bbT^m})\rightarrow \bbR$ is proper and the push-forward measures appearing in the asymptotic expansion of $\mu_{\frac{\alpha}{h}}$ are well defined.

The paper is organized as follows. In \S\S \ref{sec:classred} and \ref{sec:quantum_reduc}, we give some background. We describe the reduced cotangent space $(T^*X)_\alpha$ and its identification with $T^*(X_0/{\bbT^m})$. This allows us to get an expression for the volume form in  $T^*(X_0/{\bbT^m})$ that we will use in applications in \S \ref{sec:inverse_spectral_results}.  In \S \ref{sec:counting} we explain how the measure defined in \eqref{spectral_measure} can be given by the trace of a semi-classical operator. This is rather standard and should be well understood by experts. Section \ref{sec:semi-classical} contains the proof of our asymptotic expansion; we use a well known asymptotic formula for the Schwartz kernel and the Lemma of Stationary Phase. This is very much in the spirit of \cite{gs1}, and the reader is encouraged to consult \cite{gs1} for more details. In \S \ref{sec:inverse_spectral_results}, we apply our analytic result to get our inverse spectral results. Section \ref{weighted} makes use of the terms arising from the second line of our asymptotic expansion \eqref{asymp_expansion} to obtain inverse spectral results on weighted projective spaces using an asymptotic equivariant spectrum.\\

{\noindent \textbf{Acknowledgements.}  The authors would like to thank Alejandro Uribe for useful discussions during the International Conference on Spectral Geometry held at Dartmouth College on July 19-23, 2010.  We are also grateful to the organizers of this conference for creating an environment in which the key ideas for this article began to take shape.}


\section{The reduced cotangent space}\label{sec:classred}

We start by giving some background on the symplectic geometry of the so-called reduced cotangent space, which is the symplectic quotient of the cotangent bundle by the lift of the group action. Let $X^n$ denote a manifold that admits an action $\action$ of a torus $\bbT^m$.

Given $g\in \bbT^m$, let $\action_g$ denote the diffeomorphism of $X$ induced by the action of $g$. An element $\xi$ in the Lie algebra $\mathfrak{g}=\bbR^m$ of $\bbT^m$ determines a vector field on $X$ which we denote by $V_\xi$. This vector field is defined by
\begin{equation}\label{eqn:V_xi}
V_\xi(x):=\frac{d\action_{e^{it\xi}}(x)}{dt} \Biggl\vert_{t=0}\ .
\end{equation}
The action of $\bbT^m$ on $X$ lifts to an action $\mathcal{\action}$ on $T^*X$ in the obvious way: given $g\in \bbT^m$,
$$
g.(x,v^*)=(\action_g(x),v^*\circ d\action_{g^{-1}})\in T^*_{\action_g(x)}X \ .
$$

\begin{lemma}
The action $\action$ of $\bbT^m$ on $T^*X$ is Hamiltonian with moment map $\Phi:T^*X\rightarrow {\bbR^m}$ given by
\begin{equation}\label{eqn:momap}
\langle \Phi(x,v^*),\xi \rangle=v^*(V_\xi(x)).
\end{equation}
\end{lemma}
See \cite[II.26]{gs1} for a proof of this classical fact.  One may also consider, for each $x\in X$, the map
$$
{\bbR^m} \rightarrow T_xX,\,\,  \xi \mapsto V_\xi(x);
$$
the moment map $\Phi_{|T_x^*X}$ is the transpose of this map.

Given a Hamiltonian group action on a symplectic {manifold}, one may construct a {\it symplectic quotient}. More precisely, if ${\bbT^m}$ acts on a symplectic {manifold} $(M,\omega)$ with moment map $\phi$, for every regular value $\alpha$ of the moment map we can define
$$
M_\alpha=M//_\alpha \ {\bbT^m}:=\phi^{-1}(\alpha)/{\bbT^m}.
$$
Because the moment map is ${\bbT^m}$-invariant this quotient is well-defined. It carries a symplectic form $\omega_\alpha$ satisfying
$$
\text{pr}^*\omega_\alpha=\iota^*\omega,
$$
where $\text{pr}:\phi^{-1}(\alpha)\rightarrow \phi^{-1}(\alpha)/{\bbT^m}$ is the projection and $\iota:\phi^{-1}(\alpha)\rightarrow M$ is the inclusion. In our case since the action of $\bbT^m$ on $T^*X$ is Hamiltonian, we have a symplectic quotient
$$
(T^*X)_\alpha:=T^*X//_\alpha \ \bbT^m=\Phi^{-1}(\alpha)/\bbT^m \ .
$$
This space will play a crucial role in our results, and we will refer to it as the \emph{reduced cotangent space}.

It is clear from the definition of symplectic quotient that we must understand the regular values of the moment map.  We characterize these regular values in our setting. We will use the following fact about regular values of moment maps, a proof of which can be found in  \cite[II.26]{gs1}.
\begin{lemma}\label{lemma:regmomentmap}
Given a Hamiltonian action of a Lie group $G$ on a symplectic manifold $M$ with moment map $\Phi$, $\alpha\in \mathfrak{G}^*$ is a regular value of $\Phi$ if and only if the action of $G$ is locally free at points $p\in \Phi^{-1}(\alpha)$.
\end{lemma}

\begin{lemma}\label{lemma:regvalues}
Let $\Phi:T^*X\rightarrow {\bbR^m}$ be the moment map of the Hamiltonian action of $\bbT^m$ on $T^*X$. Then $\alpha\in {\bbR^m}$ is a regular value of $\Phi$ if and only if $\Phi^{-1}(\alpha)\subset T^*X_0$, where $X_0$ is the set of points of $X$ where $\bbT^m$ acts locally freely.
\end{lemma}
\begin{proof}
Let $p=(x, \xi)$ be a point in $\Phi^{-1}(\alpha)$.  If $x$ is in $X_0$, then the isotropy group $G_x$ of $x$ is finite.   Since the isotropy group $G_p$ of $p$ is contained in $G_x$, it too is finite.  Hence by Lemma \ref{lemma:regmomentmap}, $\alpha$ is a regular value of $\Phi$.

Now suppose $x \in X \setminus X_0$.  Then $G_x$ is a torus of positive dimension.  Moreover, $G_x$ acts linearly on $T^*_x$. The moment map $\Phi \vert_{T_x^*}$ is $G_x$-invariant because our Lie group is abelian. Therefore we can average $\xi$ by this $G_x$ action to get a $\xi$ that is also $G_x$-invariant and continues to satisfy $\Phi(x,\xi) = \alpha$.  Thus $p= (x, \xi)$ has isotropy group $G_p=G_x$ and hence by Lemma \ref{lemma:regmomentmap} $\alpha$ is not a regular value of $\Phi$.
\end{proof}

Henceforth we assume that $X$ is endowed with a fixed Riemannian metric that is $\bbT^m$-invariant. The metric allows us to split $T_xX$ as $T_x\bbT^m.x$ and its orthogonal complement, where $\bbT^m.x$ denotes the $\bbT^m$-orbit of $x$:
\begin{equation}\label{eqn:splitting}
T_xX=T_x\bbT^m.x\oplus (T_x\bbT^m.x)^\perp .
\end{equation}
{Let $\mathfrak{stab}_x$ denote the Lie algebra of the stabilizer group of $x$.}
There is a natural identification between $T_x\bbT^m.x$ and {$\bbR^m/\mathfrak{stab}_x$} given by
$$
\xi \,\text{mod}\,{\mathfrak{stab}_x}\rightarrow V_\xi(x),
$$
so $\alpha\in {\bbR^m}$ determines an element in $T^*_xX$ via the formula
\begin{equation}\label{eqn:alphasharp}
\alpha^\sharp_x(V_\xi(x)+v)=\alpha(\xi),
\end{equation}
where $v\in (T_x\bbT^m.x)^\perp$. 
The condition that $(x,v^*) \in \Phi^{-1}(\alpha)$ can be expressed as
$$
v^*={\alpha^\sharp_x} \text{ on } T_x\bbT^m.x \ .
$$

As we mentioned in the introduction, our approach is related to that of Br\"uning and Heintze \cite{bh1}. These authors think of the Laplacian restricted to equivariant functions as being related to the Laplacian of a line bundle on $X/G$, but they need to take into account the fact that the action of $G$ is not free. We instead associate to the Laplacian on equivariant functions an object on $(T^*X)_\alpha$. As the following lemma shows, the two approaches are not that different.  Note that for these authors $X_0$ denotes the subset of $X$ where the action of $G$ is free. In our setting $X_0$ denotes the subset of $X$ where the action is locally free.

\begin{lemma}\label{lemma:ident}
Let $\bbT^m$ act by isometries on $(X,g)$. There is an identification of $(T^*X)_\alpha$ with $T^*(X_0/\bbT^m)$.
 \end{lemma}
 \begin{proof}
 For any $x\in X$, the metric gives a splitting of $T_xX$ as in \eqref{eqn:splitting}.   Let $\pr:X\rightarrow X/\bbT^m$ be the usual projection. The space $T_{\pr (x)}^*(X/\bbT^m)$ can be identified with
$$
\{v^*\in T^*_xX: v^* \vert_{T_x\bbT^m.x}=0\}.
$$

Let $\alpha$ be a regular value of $\Phi$.   We can define $\alpha^\sharp$ as in \eqref{eqn:alphasharp}, and we have $(x,\alpha^\sharp_x)\in \Phi^{-1}(\alpha)$. This gives a map
$$
\begin{aligned}
\pr^*{(T^*(X_0/\bbT^m))} &\rightarrow  \Phi^{-1}(\alpha)\\
(x,v^*)& \mapsto (x,v^*+\alpha^\sharp_x)
\end{aligned}
$$
where $X_0$ is the set of points in $X$ where $\bbT^m$ acts locally freely.  
Since 
\[
\langle \Phi(x,v^*+\alpha^\sharp_x),\xi\rangle=(v^*+\alpha^\sharp_x)(V_\xi(x))=\alpha(\xi), 
\]
this map is well-defined. Given $(x,w^*)$ in $\Phi^{-1}(\alpha)$, $w^*-\alpha^\sharp_x$ vanishes on $T_x\bbT^m.x$, so the map is surjective.  It is clearly injective. Since $(\pr^*(T^*(X_0/\bbT^m)))/\bbT^m =T^*(X_0/\bbT^m)$, the result follows.
\end{proof}

We now study the symplectic structure on $T^*(X_0/\bbT^m)$ under the identification in Lemma \ref{lemma:ident}. Since $X_0/\bbT^m$ is an {orbifold}, $T^*(X_0/\bbT^m)$ has a canonical form which we denote by $\omega_0$. Given the metric $g$ on $X$ and the orthogonal splitting \eqref{eqn:splitting} at any point $x\in X_0$, we have defined $\alpha^\sharp$ by \eqref{eqn:alphasharp}.  That is, $\alpha^\sharp$ is a $1$-form on $X_0$ satisfying $V_\xi \lrcorner \alpha^\sharp=\alpha(\xi)$ for any $\xi\in {\bbR^m}$. By Cartan's magic formula,
$$
V_\xi \lrcorner d\alpha^\sharp=\mathcal{L}_{V_\xi} \alpha^\sharp-d(V_\xi \lrcorner \alpha^\sharp);
$$
we will show that both terms on the right side of this equation equal $0$.  Since $V_\xi \lrcorner \alpha^\sharp = \alpha(\xi)$ is constant, it is clear that the second term equals $0$.  For the first term, recall that $V_\xi$ is defined by \eqref{eqn:V_xi}, so that $\mathcal{L}_{V_\xi} \alpha^\sharp = \frac{d}{dt} (\action_{e^{it\xi}}^*\alpha^\sharp) \vert_{t=0}$.  We claim that $\action_{e^{it\xi}}^*\alpha^\sharp = \alpha^\sharp$, which would imply that this Lie derivative equals $0$.  At a point $x \in X_0$, we have 
\begin{eqnarray*}
(\action_{e^{it\xi}}^*\alpha^\sharp)_x (V_\xi(x) + v)&=& \alpha^\sharp_{e^{it\xi}.x} \circ d\action_{e^{it\xi}}(V_\xi(x) + v) \\
&=& \alpha^\sharp_{e^{it\xi}.x}(V_\xi(e^{it\xi}.x) + d\action_{e^{it\xi}}(v)).
\end{eqnarray*}
Since $\action_{e^{it\xi}}$ is an isometry and $v \perp V_\xi$ for all $t$, we know that $d\action_{e^{it\xi}}(v) \in (T_x\bbT^m.x)^\perp$ for all $t$.  Thus $(\action_{e^{it\xi}}^*\alpha^\sharp)_x (V_\xi(x) + v) = \alpha(\xi)$ and our claim follows.
Hence $V_\xi \lrcorner d\alpha^\sharp = 0$, which implies that $d \alpha^\sharp$ actually comes from a form on $X_0/\bbT^m$.
Let us call that form $\nu_\alpha$ and consider the pullback of $\nu_\alpha$ to $T^*(X_0/\bbT^m)$, which we denote by $\nu^\sharp_\alpha$.

\begin{lemma}
Let $\alpha$ be any regular value of $\Phi$. Under the identification $(T^*X)_\alpha\simeq T^*(X_0/\bbT^m)$ given by Lemma \ref{lemma:ident}, the reduced symplectic form on $(T^*X)_\alpha=\Phi^{-1}(\alpha)/\bbT^m$ gets mapped to $\omega_0-\nu^\sharp_\alpha$.
\end{lemma}
\begin{proof}
The identification $\pr^*(T^*(X_0/\bbT^m)) \simeq \Phi^{-1}(\alpha)$ is given by the restriction to $\pr^*(T^*(X_0/\bbT^m))$ of the map
$$
\begin{aligned}
\gamma_\alpha:T^*X_0 &\rightarrow  T^*X_0\\
(x,v^*)& \mapsto (x,v^*+\alpha^\sharp_x).
\end{aligned}
$$
The map $\gamma_\alpha$ is $\bbT^m$-equivariant 
and we will show that
$$
\gamma_\alpha^*\lambda=\lambda+\pi^*\alpha^\sharp ,
$$
where $\lambda$ is the tautological $1$-form on $T^*X_0$ and $\pi$ is the projection from $T^*X_0$ to $X_0$. We have
\[
 \gamma_\alpha^*\lambda_{(x,v^*)}=\lambda_{(x,v^*+\alpha_x^\sharp)}\circ d\gamma_\alpha=(v^*+\alpha_x^\sharp)\circ d\pi\circ d\gamma_\alpha .
 \]
Since $\pi\circ\gamma_\alpha=\pi$, the preceding equation becomes 
\[
\gamma_\alpha^*\lambda_{(x,v^*)}=v^*\circ d\pi+\alpha_x^\sharp\circ d\pi=\lambda_{(x,v^*)}+\pi^*\alpha^\sharp_{(x,v^*)}. 
\]
This implies that 
$$
\gamma_\alpha^*\omega_0=\omega_0-\pi^*d\alpha^\sharp ,
$$
where $\omega_0=-d\lambda$ is the canonical symplectic form on $T^*X_0$. Restricting $\gamma_\alpha$ to $\pr^*(T^*(X_0/\bbT^m))$ and taking the quotient by $\bbT^m$ we get a diffeomorphism
$$
g_\alpha:T^*(X_0/\bbT^m)\rightarrow (T^*X)_\alpha .
$$
Letting $\omega_\alpha$ be the symplectic form obtained by symplectic reduction on $(T^*X)_\alpha$, we have
$$
g_\alpha^*\omega_\alpha=\omega_0-\nu_\alpha^\sharp .
$$
This relation follows directly from the corresponding relation for $\gamma_\alpha$ and the definition of reduced symplectic form.
\end{proof}

We conclude this section with a discussion of symplectic reduction in stages.  We want to identify $(T^*X)_\alpha$, which was obtained via symplectic reduction of $T^*X$ with respect to $\bbT^m$ at level $\alpha \in {\bbR^m}$, with a space obtained by doing symplectic reduction twice.  Recall that $\chi_\alpha$ is a character of $\bbT^m$ associated with $\alpha$, defined by $\chi_\alpha(e^{i\theta}) = e^{i\theta \cdot \alpha}$.  Letting $K$ be the kernel of the homomorphism $\chi_\alpha: \bbT^m \rightarrow S^1$ and noting that $\chi_\alpha$ is surjective, 
we have an isomorphism $\bbT^m/K \cong S^1$.  
Thus $Y:= X_0/K$ can be viewed as an $S^1$-bundle over $X_0/\bbT^m$, denoted $\text{proj}: Y \rightarrow X_0/\bbT^m$.
  
We may perform symplectic reduction of $T^*X$ with respect to $K$ at level zero.  The Hamiltonian action of $\bbT^m$ on $T^*X$ restricts to a Hamiltonian action of $K$ on $T^*X$, and there is a corresponding restricted moment map $\bar{\Phi}: T^*X \rightarrow \mathpzc{k}^*$. 
Since $\alpha \vert_{\mathpzc{k}}$ is the zero map, we see that $(T^*X)_\alpha$ can be obtained by reducing $T^*X$ with respect to $K$ at level $0$, and then reducing the resulting space with respect to $S^1$ at moment level $1$.
We will use this construction for $(T^*X)_\alpha$, which is called \emph{reduction in stages} \cite[p. 149]{cds}, in the next section.


\section{Quantum reduction}\label{sec:quantum_reduc}

{Using the same notation as in \S \ref{sec:classred}, we will now} describe the operator $h^2\Lap_{|\Sm(X)^\frac{\alpha}{h}}$, viewed as a semi-classical differential operator on $X_0/\bbT^m$.  Note that the reduction in stages carried out in the previous section can also be performed in the quantum setting.  Since the Riemannian metric on $X$ is $\bbT^m$-invariant, it is automatically $K$-invariant; hence this metric induces a metric on $Y$ and the operator $h^2\Lap_{|\Sm(X)^\frac{\alpha}{h}}$ can be identified with the operator $h^2\Lap_{|\Sm(Y)^\frac{1}{h}}$, where $\Sm(Y)^\frac{1}{h}$ is the space of functions in $\Sm(Y)$ that satisfy the transformation law $f(e^{i\theta}y) = e^{\frac{i\theta}{h}}f(y)$. 

Before venturing into this quantum setting, we will describe the classical differential operator $\Lap_{|\Sm(Y)^1}$, where $\Sm(Y)^1$ is the space of functions in $\Sm(Y)$ that satisfy the transformation law $f(e^{i\theta}y) = e^{i\theta}f(y)$.  Since $Y$ is equipped with a Riemannian metric, we get an orthogonal splitting of $T_yY$ analogous to that in \eqref{eqn:splitting}:
\begin{equation}\label{eqn:splitting2}
T_yY = T_yS^1.y \oplus (T_yS^1.y)^\perp .
\end{equation}
The vector field $\frac{\partial}{\partial \theta}$ on $S^1$ gives a trivialization of the first summand in \eqref{eqn:splitting2}, so we may replace $T_yS^1.y$ by $\bbR$.  
Recall that $Y= X_0/K$ can be viewed as an $S^1$-bundle over $X_0/\bbT^m$, denoted $\text{proj}: Y \rightarrow X_0/\bbT^m$. {There is an identification
$$
\begin{aligned}
(T_yS^1\cdot y)^\perp&\mapsto \text{proj}^*(T_{\operatorname{proj}(y)}^*(X_0/\bbT^m))\\
v&\rightarrow d \text{proj}_y(v).
\end{aligned}
$$
 Thus \eqref{eqn:splitting2} gives rise to a dual splitting
\begin{equation}\label{eqn:dualsplit}
T_y^*Y \cong \bbR \oplus \operatorname{proj}^*(T_{\operatorname{proj}(y)}^*(X_0/\bbT^m)).
\end{equation}}

The vector $1$ in the first summand of \eqref{eqn:dualsplit} gives rise to a section $\mu$ of $T^*Y$.  Note that $\mu$ is dual to the vector field $\frac{\partial}{\partial \theta}$, so we have
\[
\frac{\partial}{\partial \theta} \lrcorner \mu = 1.
\]
Hence $\mu$ defines a connection form on the $S^1$-bundle $\text{proj}: Y \rightarrow X_0/\bbT^m$. 
 In particular, $\mu$ possesses a curvature form $\nu \in \Omega^2(X_0/\bbT^m)$ satisfying $\operatorname{proj}^*\nu = d\mu$.

We can also view the space $C^\infty(Y)^1$ in terms of objects defined on $X_0/\bbT^m$.  Fix $p \in X_0/\bbT^m$, and let $Y_p$ be the fiber of $Y$ above $p$.  Denote
\[
\mathbb{L}_p = \{ f: Y_p \rightarrow \bbC \ \vert \  f(e^{i\theta}y) = e^{i\theta}f(y)\} .
\]
The assignment $p \rightarrow \mathbb{L}_p$ defines a line bundle over $X_0/\bbT^m$. Using our definition of $C^\infty(Y)^1$, one can check that sections of this line bundle correspond to elements of $C^\infty(Y)^1$.  Hence we may identify $C^\infty(Y)^1$ with $C^\infty(\mathbb{L})$.

We claim that the connection on $Y$ gives us a connection on $\mathbb{L}$.  Suppose $W$ is a vector field on $X_0/\bbT^m$, with $\widetilde{W}$ its horizontal lift to $Y$.  The Lie differentiation operator $L_{\widetilde{W}}: C^\infty(Y) \rightarrow C^\infty (Y)$ commutes with the $S^1$-action and hence maps $C^\infty(Y)^1$ into $C^\infty(Y)^1$.  Via our identification of $C^\infty(Y)^1$ with $C^\infty(\mathbb{L})$, we see that $L_{\widetilde{W}}$ defines a covariant differentiation operator $\nabla_W: C^\infty(\mathbb{L}) \rightarrow C^\infty(\mathbb{L})$.  
To see what $\nabla_W$ looks like locally, let $(U, x_1, \dots, x_n)$ be a coordinate patch on $X_0/\bbT^m$ and $(\widetilde{U}, x_1, \dots, x_n, \theta)$ the coordinate patch sitting above $U$ in $Y$.  In these coordinates, we may write the section $\mu$ of $T^*Y$ introduced above as 
\[
\mu = \sum \mu_k(x) dx_k + d\theta \ .
\]
Let $\nu_k$ be the horizontal lift of $\frac{\partial}{\partial x_k}$; horizontality implies that $\nu_k \lrcorner \mu = 0$, and thus 
\[
\nu_k = \frac{\partial}{\partial x_k} - \mu_k(x) \frac{\partial}{\partial \theta} .
\]
Hence given a function $\tilde{f} = f(x)e^{i\theta}$ in $C^\infty(\widetilde{U})^1$, we have
\begin{equation}\label{eqn:Liederiv}
L_{\nu_k}\tilde{f} = \Bigg( \frac{\partial}{\partial x_k} - \sqrt{-1} \mu_k(x) \Bigg) f(x)e^{i\theta},
\end{equation}
since the Lie derivative of a function $f$ along a vector field $\nu$ is the evaluation $\nu(f)$.
Therefore, via the identification of $C^\infty({U})^1$ with $C^\infty(\mathbb{L} \vert_{\widetilde{U}})$ that matches $f$ with $fe^{i\theta}$, we may write
\[
\nabla_{\frac{\partial}{\partial x_k}}= \frac{\partial}{\partial x_k} - \sqrt{-1} \mu_k .
\]
This gives a local description of a connection on $\mathbb{L}$ arising from the connection on $Y$.

Using these ingredients, we describe the operator $\Lap \vert_{\Sm(Y)^1}$ as an operator on $X_0/\bbT^m$.  The $\nu_k$ will serve as coordinate vectors for the chart on $\widetilde{U}$ obtained via horizontal lifts. Then on $\widetilde{U}$ we may write
\[
\Lap = - \Bigg( \sum_{k,\ell=1}^{n} g^{k,\ell} L_{\nu_k} L_{\nu_\ell} + \sum_{k=1}^{n} g^{k} L_{\nu_k} + g^{00} \frac{\partial^2}{\partial \theta^2} \Bigg),
\]
where the $g^{ij}$'s are the entries of the inverse of the matrix of the metric in the frame $\{\nu_i\}$. 
The absence of $g^{0,k}$ terms is due to the fact that the vector fields $\nu_k$ are perpendicular to the vector field $\frac{\partial}{\partial \theta}$.  Note that $\frac{\partial}{\partial \theta}(f) = f$ for $f \in \Sm(\widetilde{U})$. Restricting to functions in $\Sm(\widetilde{U})^1$, we may write
\[
\Lap = - \Bigg( \sum_{k,\ell=1}^{n} g^{k,\ell} \Big(\frac{\partial}{\partial x_k} - \sqrt{-1} \mu_k\Big)\Big(\frac{\partial}{\partial x_\ell} - \sqrt{-1} \mu_\ell\Big) + \sum_{k=1}^n g^{k} \Big(\frac{\partial}{\partial x_k} - \sqrt{-1} \mu_k\Big) \Bigg) + V ,
\]
where $V$ is the potential function $g_{00}(y) = \langle \frac{\partial}{\partial \theta}, \frac{\partial}{\partial \theta} \rangle_p$.  Thus the operator on $X_0/\bbT^m$ corresponding to $\Lap \vert_{\Sm(Y)^1}$ is the Schr\"{o}dinger operator $\Lap_\mathbb{L} + V$, where
\begin{equation}\label{eqn:Schro}
\Lap_\mathbb{L} = - \Bigg( \sum_{k,\ell=1}^{n} g^{k,\ell} \nabla_{\frac{\partial}{\partial x_k}} \nabla_{\frac{\partial}{\partial x_\ell}} + \sum_{k=1}^{n} g^{k} \nabla_{\frac{\partial}{\partial x_k}} \Bigg) .
\end{equation}

Next we venture into the semi-classical setting and describe the semi-classical operator $h^2\Lap$ on the space $\Sm(Y)^{\frac{1}{h}}$.  Recall that $h^{-1}:= N$ is an integer, and that $\Sm(Y)^{\frac{1}{h}}$ is the space of functions in $\Sm(Y)$ satisfying the transformation law $f(e^{i\theta} y) = e^{\frac{i\theta}{h}}f(y)$.  We want to identify elements of $\Sm(Y)^{\frac{1}{h}}$ with sections of a certain line bundle.   As above, we fix $p \in X_0/\bbT^m$, and let $Y_p$ be the fiber of $Y$ above $p$.  Denote
\[
\mathbb{L}_p^N = \{ f: Y_p \rightarrow \bbC \ \vert \  f(e^{i\theta}y) = e^{iN\theta}f(y)\}.
\]
The assignment $p \rightarrow \mathbb{L}_p^N$ defines a line bundle over $X_0/\bbT^m$.  Arguing as for $h=1$, one can check that sections of this line bundle correspond to elements of $C^\infty(Y)^{\frac{1}{h}}$. 
 Hence we may identify $C^\infty(Y)^{\frac{1}{h}}$ with $C^\infty(\mathbb{L}^N)$.

The other main difference from the classical setting stems from the expression for the Lie differentiation operator.  As above, we let $(U, x_1, \dots, x_n)$ be a coordinate patch on $X_0/\bbT^m$ and $(\widetilde{U}, x_1, \dots, x_n, \theta)$ the coordinate patch sitting above $U$ in $Y$.  Then elements of $\Sm(Y)^{\frac{1}{h}}$ are functions of the form $\tilde{f} = f(x)e^{\frac{i\theta}{h}}$, and \eqref{eqn:Liederiv} is replaced by
\[
L_{\nu_k}\tilde{f} = \Bigg( \frac{\partial}{\partial x_k} - \frac{\sqrt{-1}}{h}\mu_k(x) \Bigg) f(x)e^{\frac{i\theta}{h}} .
\]

Now let $(\xi_1, \ldots, \xi_n, \eta)$ be the dual cotangent coordinates to $(x_1, \dots, x_n, \theta)$.  Then the semi-classical symbol of the semi-classical differential operator $\frac{h}{\sqrt{-1}}L_{\nu_k} \vert_{\Sm(Y)^{\frac{1}{h}}}$ is $\xi_k - \mu_k(x)$.
This symbol has a nice interpretation in terms of the symbol of the operator $\frac{h}{\sqrt{-1}}L_{\nu_k} \vert_{\Sm(Y)}$.  Recalling that $\nu_k = \frac{\partial}{\partial x_k} - \mu_k(x) \frac{\partial}{\partial \theta}$, we see that the symbol of $\frac{h}{\sqrt{-1}}L_{\nu_k}$, viewed as a semi-classical operator, is $\xi_k - \mu_k(x)\eta$.
However, the moment map $\phi: T^*Y \rightarrow \bbR$ associated with the $S^1$-action on $T^*Y$ is given in these coordinates by $\phi(x, \theta, \xi, \eta) = \eta$, so the semi-classical symbol of $\frac{h}{\sqrt{-1}}L_{\nu_k} \vert_{\Sm(Y)^{\frac{1}{h}}}$ is the symbol of $\frac{h}{\sqrt{-1}}L_{\nu_k}$ restricted to $\phi^{-1}(1)$.  The same reasoning shows that the semi-classical symbol of the Laplace operator $h^2 \Lap \vert_{\Sm(Y)^{\frac{1}{h}}}$ is the restriction to $\phi^{-1}(1)$ of the symbol of $h^2\Lap \vert_{\Sm(Y)}$. 

There is another, slightly more intrinsic, way to view the symbol of $h^2 \Lap \vert_{\Sm(Y)^{\frac{1}{h}}}$.  Let ${\fp_0}$ be the {leading} semi-classical symbol of the operator $h^2 \Lap \vert_{\Sm(Y)}$.  Since ${\fp_0} \vert_{\phi^{-1}(1)}$ is $S^1$-invariant, it can be viewed as the pullback to $\phi^{-1}(1)$ of a function {$\fp_1$} on the reduced space $(T^*Y)_1$.  Thus we can think of the symbol of $h^2 \Lap \vert_{\Sm(Y)^{\frac{1}{h}}}$ as being {$\fp_1$}.  By reduction in stages {(e.g., \cite{sjaler}), this $\fp_1$ is the same as $\fp_\alpha$ with $\alpha=1$} discussed in the introduction.

Finally, we note that the identification of $\Sm(Y)^{\frac{1}{h}}$ with $\Sm(\mathbb{L}^N)$ allows us to convert the operator $h^2 \Lap \vert_{\Sm(Y)^{\frac{1}{h}}}$ into an operator on $\Sm(\mathbb{L}^N)$.  When $h=1$, we showed above that the corresponding operator on $\Sm(\mathbb{L})$ is the operator $\Lap_{\mathbb{L}} + V$, where $\Lap_{\mathbb{L}}$ is as in \eqref{eqn:Schro}.  The same argument shows that for arbitrary $h = \frac{1}{N}$, the corresponding operator on $\Sm(\mathbb{L}^N)$ is the semi-classical Schr\"{o}dinger operator $h^2 \Lap_{\mathbb{L}} + V$.


\section{Counting equivariant eigenfunctions}\label{sec:counting}

Our goal in this section is to give some intuition and motivation for the constructions and proofs in \S \ref{sec:semi-classical}.  {We will be somewhat informal in our presentation.}  We begin by describing some basic notions in semi-classical analysis.

Semi-classical analysis is concerned with families of differential operators. For $h \in \RR_+$ and {$X^n$ a Riemannian}  manifold, let $\op_h :
\Sm(X) \to \Sm(X)$ be a self-adjoint $p${th} order
differential operator depending smoothly on $h$.  In
semi-classical analysis $\op_h$ is an operator of ``order zero''
if, locally on coordinate patches,
\begin{equation*}
  \op_h = \sum_{|\alpha | \leq p} a_\alpha (x,h) h^{|\alpha|}
      D^\alpha ,
\end{equation*}
where $D^\alpha = D^{\alpha_1}_1 \cdots D^{\alpha_d}_d$, $D_i = -
\sqrt{-1} \frac{\partial}{\partial x_i}$, $|\alpha| = \alpha_1 + \dots + \alpha_d$, and $a_\alpha$ is a
$\Sm $ function of $x$ and $h$.  The symbol of this operator is  
\begin{equation*}
\fp (x,\xi,h) = \sum_{|\alpha | \leq p} a_{\alpha} (x,h) \xi^{|\alpha|}
\end{equation*}
{and its leading symbol ${\fp_0} : 
 T^* X \to \RR$ is the function}
 \begin{equation*}
   \fp_0 (x,\xi) =\sum_{|\alpha| \leq p} a_\alpha (x,0)
      \xi^{{|\alpha|}} ,
 \end{equation*}
that is, $\fp_0 (x,\xi) = p(x,\xi,0)$.  We assume that $\fp_0$ satisfies the condition $|\fp_0 (x, \xi)| \geq c_k || \xi ||^k + o(|| \xi ||^k)$ for some $0 \leq k \leq p$ and positive constant $c_k$.  See \cite[Chap. 10]{gs} for more details.

Our discussions will focus on the following examples:
\begin{enumerate}
\item 
the semi-classical Laplace operator $\op_h= {h^2 \Delta_X}$ acting on smooth functions on a manifold $X$, with symbol $\fp(x,\xi{,h}) = {||\xi||}^2_x$ and with leading semi-classical symbol $\fp_0 (x,\xi) = {||\xi||}^2_x$; 
\item the Schr\"odinger operator $\op_h =-h^2 \sum \frac{\partial^2}{\partial x^2_i} +V (x)$ acting on smooth functions on $\bbR^n$ with potential V, with symbol $\xi^2 + V(x)$. {Note that, in order for $\fp_0$ to satisfy the condition above, we must have $V(x) \rightarrow \infty$ as $|x| \rightarrow \infty$.}
\end{enumerate}

By the Schwartz kernel theorem (e.g., \cite[p. 10]{treves}), the operator $\op_h$ has a kernel $K_h(x,y)$ given by 
$$
\op_h f(y)=\int_XK_h (x,y)f(x)dx.
$$
Let $\lambda_{i,h}, i=1,2,\dots$, be the discrete set of eigenvalues of $\op_h$ with corresponding eigenbasis $\{f_{i,h}\}$.
Since $\op_h$ is self-adjoint, we can choose our eigenbasis to be orthonormal.  In fact, $\{f_{i,h}\}$ is an orthonormal basis for $\mathcal{L}^2(X)$, allowing us to formally express our kernel as
$$
K_h(x,y)=\sum_i \lambda_{i,h} f_{i,h}(x)\bar{f}_{i,h}(y).
$$
If $\op_h$ had a well-defined trace, that is, if the series $\sum_i \lambda_{i,h}$ were convergent, we would have
$$
\tr(\op_h)=\sum_i \lambda_{i,h}=\int_X K_h(x,x) {dx}.
$$
However, the series $\sum_i \lambda_{i,h}$ generally does not converge.  

One way of circumventing this difficulty is to localize $\op_h$ using a compactly supported function $\rho \in C_0^{\infty}(\bbR)$. Define the operator $\rho(\op_h)$ by setting
$$
\rho(\op_h) f_{i,h}={\rho(\lambda_{i,h})} f_{i,h},\, i=1, 2, \dots .
$$
Although $\rho(\op_h)$ is not necessarily a semiclassical differential operator, it is a zeroth order semiclassical pseudodifferential operator \cite[Chap. 8]{DiSjo}. \footnote{
Note that Dimassi and Sj\"{o}strand are working in $\bbR^n$ and not on a general manifold; to pass from $\bbR^n$ to an arbitrary manifold $X$, we use coordinate charts: we require that the kernel of $\rho(\op_h)$ be smooth off the diagonal in $X \times X$, and that $\rho(\op_h)$ have the correct form on coordinate neighborhoods.  Namely, let $U\subset X$ be a coordinate neighborhood with coordinate chart $\beta: U \rightarrow O$, where $O$ is an open subset of $\mathbb{R}^n$.  We require that the map of $C_0^{\infty}(O)$ into $C^{\infty}(O)$ defined by $u \mapsto \rho(\op_h)(u \circ \beta) \circ \beta^{-1}$ be a semiclassical pseudodifferential operator.}  In particular, $\rho(\op_h)$ is pseudolocal, which means that if $\varphi, \psi \in C^{\infty}_0 (X)$ have non-overlapping support then $\langle \rho(\op_h) \varphi ,\psi \rangle$ vanishes to infinite order in $h$. 

If we make an additional assumption, we can conclude that $\rho(\op_h)$ is trace class. 
In particular, suppose that for some interval $[a,b] \subset \RR$ and some
$\epsilon >0$,\break $\fp_0^{-1}$ $ ([a-\epsilon , b+\epsilon ])$ is
compact.  Then by the semi-classical version of the classical Friedrichs' theorem {(see \cite[\S 13.6]{gs}),} we have
\begin{equation*}
  \spec (\op_h)\cap [a,b]  = \{ \lambda_{i,h}\, , \, 1\leq i \leq N(h)\},
\end{equation*}
where
\begin{equation*}
 N (h) \sim (2\pi h)^{-n} \vol ({\fp_0}^{-1} ([a,b])) \, .
\end{equation*}
For $\rho \in C^{\infty}_0 (a,b)$, the trace of $\rho(\op_h)$ is well defined and is given by
$$   \tr (\rho (\op_h)) = \sum^{N(h)}_{i=1} \rho
      (\lambda_{i,h}).$$
The operator $\rho(\op_h)$ admits a Schwartz kernel $K_{\rho,h}(x,y)$ given by
$$
K_{\rho,h}(x,y)=\sum_{i=1}^{N(h)} \rho(\lambda_{i,h})f_{i,h}(x){\bar{f}_{i,h}}(y),
$$
so we may also express the trace as
$$
\tr(\rho(\op_h))=\int_X K_{\rho,h}(x,x)dx.
$$

Now assume we have a diffeomorphism $F$ of $X$.  We define 
$$
\tr (F_*\op_h):=\int_X{K_h(F(x),x)}dx ,
$$
when such an integral converges. Let $F^\sharp:\mathcal{L}^2(X) \rightarrow \mathcal{L}^2(X)$ be the operator defined by composing with $F$, i.e., $F^\sharp(f)=f\circ F$, and suppose $F^\sharp$ commutes with $\op_h$ so that $F^\sharp$ preserves the eigenspaces of $\op_h$.  Then 
\begin{equation}\label{eqn:traceFsharp}
\tr(F_*\rho(\op_h))=\int_X K_{\rho,h}(F(x),x)dx=\int_X\sum_{i=1}^{N(h)} \rho(\lambda_{i,h}) f_{i,h}(x)F^\sharp \bar{f}_{i,h}(x)dx.
\end{equation}

To write this trace in a slightly different form, we introduce some notation.  Let $E_{\lambda}$ denote the eigenspace corresponding to the eigenvalue $\lambda \in \operatorname{Spec}(\op_h)$, and let $n_\lambda=\dim(E_\lambda)$. Also let $F^\sharp_\lambda$ denote the restriction of the operator $F^\sharp$ to $E_\lambda$. Then \eqref{eqn:traceFsharp} can be written as 
$$
\tr(F_*\rho(\op_h))=\sum_{i=1}^{N(h)} \rho(\lambda_{i,h})\sum_{j=1}^{n_{\lambda_{i,h}}} \int_X f_{j,h}(x)F^\sharp_{\lambda_{i,h}} \bar{f}_{j,h}(x)dx,
$$
or equivalently,
$$
\tr(F_*\rho(\op_h))=\sum_{i=1}^{N(h)} \rho(\lambda_{i,h})\sum_{j=1}^{n_{\lambda_{i,h}}} \langle f_{j,h}(x),F^\sharp_{\lambda_{i,h}} f_{j,h}(x)\rangle=
\sum_{i=1}^{N(h)} \rho(\lambda_{i,h}) \tr(F^\sharp_{\lambda_{i,h}}).
$$
Note that $\tr(F^\sharp_\lambda)$ does not count the $F$-invariant eigenfunctions, since $F^\sharp_\lambda$ may have eigenvalues different from $1$. 

In particular, we are interested in the situation of a group $G$ acting on $X$.  Let $\action:G\rightarrow \text{Diffeo}(X)$ be the group homomorphism, and assume that the $G$ action commutes with $\op_h$.  We have 
\[
\tr(\action(g)_*\rho(\op_h))=\sum_{i=1}^{N(h)} \rho(\lambda_{i,h}) \tr(\action(g)^\sharp_{\lambda_{i,h}}),
\]
for $g\in G$ and we can integrate this with respect to the Haar measure on $G$ to obtain
$$
\int_G \tr(\action(g)_*\rho(\op_h))dg=\sum_{i=1}^{N(h)} \rho(\lambda_{i,h}) \int_G\tr(\action(g)^\sharp_{\lambda_{i,h}})dg.
$$
By complexifying, it follows from the Peter-Weyl theorem that we can write the representation of $G$ on $E_\lambda$ as a direct sum of characters (representations of $G$ on $\bbC$), i.e., 
$$
\action^\sharp_\lambda=\bigoplus_a(\chi^a_\lambda\oplus {\bar{\chi}^a_\lambda})\bigoplus n_\lambda\bbC
$$
where $\bbC$ is the trivial representation and $n_\lambda$ is the dimension of the space of $G$-invariant functions in $E_\lambda$. Therefore 
$$
\int_G \tr(\action(g)_*\rho(\op_h))dg=\sum_{i=1}^{N(h)} \rho(\lambda_{i,h})n_{\lambda_{i,h}}+\sum_a \rho({\lambda_{i,h})} \int_G (\chi^a_{\lambda_{i,h}}+ \bar{\chi}^a_{\lambda_{i,h}})dg.
$$
If $G$ is a {torus} $\bbT^m$, the characters $\chi^a_\lambda$ are given by integral weights $\alpha_a \in \bbZ^m$, i.e., 
$
\chi^a_\lambda(e^{i\theta})=e^{i\alpha_a\cdot \theta}.
$
Thus
\begin{equation}\label{eqn:charint0}
\int_G \chi^a_\lambda dg=\int_{\bbT^m}e^{i\alpha_a\cdot \theta} d\theta=0.
\end{equation}
Hence we conclude that 
$$
\int_{\bbT^m} \tr(\action(g)_*\rho(\op_h))dg=\sum_{i=1}^{N(h)} \rho(\lambda_{i,h})n_{\lambda_{i,h}},
$$
so that $\int_{\bbT^m} \tr(\action(g)_*\rho(\op_h))dg$ actually ``counts" $\bbT^m$-invariant eigenfunctions. 

Next we want to ``count" equivariant eigenfunctions, which are complex-valued eigenfunctions that satisfy
$$
f\circ \action(g)=\chi_0(g)f
$$
for some fixed character $\chi_0$. Thus we should consider
$$
\int_G \tr(\chi_0^{-1}\action(g)_*\rho(\op_h))dg=\sum_{i=1}^{N(h)} \rho(\lambda_{i,h})n^{\chi_0}_{\lambda_{i,h}}+\sum_a {\rho(\lambda_{i,h})}\int_G (\chi^a_{\lambda_{i,h}}+ \bar{\chi}^a_{\lambda_{i,h}})dg ,
$$
where $n^{\chi_0}_\lambda$ is the dimension of the space of $\chi_0$-equivariant eigenfunctions in $E_\lambda$. If $G=\bbT^m$, then {$\chi_0$ is given by some weight $\alpha \in \bbZ^m$.  So \eqref{eqn:charint0} holds and} 
$$
\int_{\bbT^m} \tr(e^{-i\alpha\cdot \theta}\action(\theta)_*\rho(\op_h))d\theta=\sum_{i=1}^{N(h)} \rho(\lambda_{i,h})n^{\alpha}_{\lambda_{i,h}},
$$
where $n^{\alpha}_{\lambda}$ is the dimension of the space of $\alpha$-equivariant eigenfunctions in $E_\lambda$ and $\action(\theta)$ stands for $\action(e^{i\theta})$. Recall that by \eqref{eqn:traceFsharp}, we may also write
\begin{equation}\label{eqn:traceaskernel}
\tr(e^{-i\alpha\cdot \theta}\action(\theta)_*\rho(\op_h))=\int_{X}e^{-i\alpha\cdot \theta}K_{\rho,h}(\action(\theta)x,x)dx. 
\end{equation}

In the next section, we will use the representation of the trace as an integral kernel given by \eqref{eqn:traceaskernel} to count equivariant eigenfunctions.  In particular, we will apply the method of stationary phase to get an asymptotic expression for our integral as $h$ tends to $0$.  

 The main point we want to stress here is that,  in order for the weight to play a non-trivial role in the asymptotic expansion of \eqref{eqn:traceaskernel}, it should also depend on $h$. This is in fact a big difference between our work and previous work concerning equivariant eigenspaces in, say, \cite{bh2}. This justifies the introduction of the notion of \emph{semi-classical weights}. We will replace $\alpha$ by $\frac{\alpha}{h}$, where we assume that $\frac{1}{h} \in \bbZ$.  Hence the {integrated} trace we will be considering is 
$${\int_{\bbT^m}}\tr(e^{-\frac{i\alpha\cdot \theta}{h}}\action(\theta)_*\rho(\op_h)) {d \theta};$$ 
as in the case of weight $\alpha$ discussed above, this ``counts" equivariant eigenfunctions of weight $\frac{\alpha}{h}$.


\section{{Asymptotic expansion of the spectral measure}}\label{sec:semi-classical}

Let $X$ be a compact $n$-dimensional manifold which admits an effective action $\action$ of a torus $\bbT^m$.  We begin by establishing some notation.
\begin{itemize}
\item Let $X_0$ be the open subset of $X$ on which this action is locally free.
\item {Let $e^{i\theta_1}, \dots, e^{i\theta_N}$ be the elements of $\bbT^m \setminus \{e\}$ in the stabilizer groups of points in $X_0$.  Note that since the action is locally free on $X_0$}, all stabilizers are discrete.
\item For each $e^{i\theta_r}$, the connected components of its fixed point set will be labelled $X_{r,s}$, where $s = 1, \dots, N_r$.  
\item Let $\op_h: \Sm(X) \rightarrow \Sm(X)$ be a $\bbT^m$-invariant family of self-adjoint $p$th order differential operators depending smoothly on $h$; viewed as semiclassical operators, $\op_h$ are of order zero. Assume that the leading symbol $\fp_0: T^*X \rightarrow \bbR$ of $\op_h$ satisfies the condition $|\fp_0(x,\xi)| \geq c_k ||\xi||^k + o(||\xi||^k)$ for some $0 \leq k \leq p$ and positive constant $c_k$.
\item Let $\frac{\alpha}{h}$ be a weight of the $\bbT^m$ action, where $\alpha$ is a regular value of the moment map $\Phi:T^*X\rightarrow \bbR^m$ and $\frac{1}{h} \in \bbZ$. {This weight defines a character $\chi_{\frac{\alpha}{h}}$ of $\bbT^m$.}
\item Consider the restriction of $\op_h$ to the space of $\frac{\alpha}{h}$-equivariant functions $\Sm(X)^{\frac{\alpha}{h}}$, and denote the corresponding eigenvalues, counting multiplicities, by $\lambda_{i,h}(\alpha)$ for $i=1, 2, \dots$. 
\end{itemize}
After localization using a function $\rho \in C_0^{\infty}(a,b)$, we can count the equivariant eigenfunctions of weight $\frac{\alpha}{h}$ via the spectral measure
\[
\mu_{\frac{\alpha}{h}}(\rho) = \sum_{i=1}^{N(h)} \rho\left(\lambda_{i,h}\left( \alpha \right) \right), 
\]
where the eigenvalues are repeated according to their multiplicities.
As we saw in \S \ref{sec:counting}, the spectral measure can also be expressed as
\begin{equation}\label{eqn:altspecmeas}
\mu_{\frac{\alpha}{h}}(\rho) = \int_{\bbT^m} \tr(e^{-\frac{i\alpha\cdot \theta}{h}}\action(\theta)_*\rho(\op_h)) d \theta = \int_{\bbT^m}\int_X e^{-\frac{i \alpha \cdot \theta}{h}} K_{\rho,h}(\action(\theta) x,x) dx d\theta .
\end{equation}

Our goal is to get an asymptotic expansion in $h$ for this measure.  Let $(T^*X)_{\alpha}$ and $(T^*X_{r,s})_{\alpha}$ be the symplectic reductions of $T^*X$ and $T^*X_{r,s}$ at $\alpha$, and let $\nu$ and $\nu_{r,s}$ be the symplectic volume forms on these spaces as defined in \S \ref{sec:classred}.  From the symbol $\fp: T^*X \rightarrow \bbR$ of $\op_h$, one gets maps
\[
\fp_{\alpha}: (T^*X)_{\alpha} \rightarrow \bbR \text{\ \ \  and \ \ \   } \fp_{\alpha}^{r,s}: (T^*X_{r,s})_{\alpha} \rightarrow \bbR.
\]
These maps will give rise to measures on $\bbR$ defined as the {pushforwards of the measures induced by the volume forms} $\nu$ and $\nu_{r,s}$ on $(T^*X)_{\alpha}$ and  $(T^*X_{r,s})_{\alpha}$, respectively.

\begin{thm}\label{mainthm}
With the setup and notation given above, 
\begin{align}
\mu_{\frac{\alpha}{h}}(\rho) &\sim& (2\pi h)^{{m-n}} \sum_{i=0}^{\infty} h^i \sum_{\ell \leq 2i} \mu_{i,\ell}{\left( \frac{d^{\ell}\rho}{dt^{\ell}} \right)} \hspace*{5cm} \notag \\ 
&+& \sum_{r,s} (2 \pi h)^{{-k}} \chi_{\frac{\alpha}{h}}(e^{-i\theta_r}){|\operatorname{det}(A_r-I)|}^{-1} \left( (\fp_{\alpha}^{r,s})_* \nu_{r,s} +O(h^{-{k}+1})\right), 
\end{align}
{where $k \geq m$ depends on $r$ and $s$, $A_r$ is the linear action of $\tau(\theta_r)$ restricted to the normal space to $X_{r,s}$,} and the measures $\mu_{i,\ell}$ are defined by
\[
\mu_{i,\ell} = (\fp_\alpha)_* b_{i,\ell} \nu
\]
where the $b_{i,\ell}$'s  are smooth functions on the corresponding reduced spaces.  Moreover, the leading term in this expansion is given by $b_{0,0} = 1$. 
\end{thm}

Note that in the course of proving this theorem, we give an explicit expression for the lower-order terms attached to the {summation over $r$ and $s$}.
To prove the theorem, we use two analytical tools. The first tool is an explicit expression for the Schwartz kernel of $\rho (\op_h)$ in local coordinates. 

 \begin{lemma}[Schwartz kernel asymptotic expansion]\label{usual_schwartz_kernel}
{With the setup and notation given above}, let $K_{\rho, h}(x,y)$ be the Schwartz kernel for $\rho(\op_h)$ in some local coordinates. Then $K_{\rho,h}(x,y)$ admits an asymptotic expansion in powers of $h$:
\begin{equation}
  \label{eq:2.8}
  (2\pi h)^{-n} \sum^\infty_{k=0}  h^k \int a_{\rho,k}
    (y,\xi) e^{\frac{i (x-y)\cdot \xi}{h}} \, d \xi ,
\end{equation}
where
\begin{equation}
  \label{eq:2.9}
  a_{\rho,k} (y,\xi) = \sum_{\ell \leq 2k} b_{k,\ell} (y, \xi)
    \left( \frac{d^\ell \rho }{ds^\ell} \right)({\fp_0} (y,\xi))
\end{equation}
and the leading order term in (\ref{eq:2.8}) is given by
$a_{\rho ,0} (y,\xi)= \rho (\fp_0 (y,\xi))$. 
\end{lemma}
We will not give a proof of this lemma here.  See \cite[pp. 102-103]{DiSjo} or \cite[Chap. 10]{gs}; as explained previously, the arguments therein apply to $\bbR^n$ and we must use coordinate charts to pass from $\bbR^n$ to an arbitrary manifold $X$.
In our setting, Lemma \ref{usual_schwartz_kernel} gives us an asymptotic expansion of the integrated trace.

The second ingredient we will use is the lemma of stationary phase (e.g., \cite[Chap. 15]{gs}). 
\begin{lemma}[Lemma of stationary quadratic phase]\label{stationary_phase}
Let $A$ be an $n \times n$ nonsingular self-adjoint matrix, and let $f \in C_0^{\infty}(\bbR^n)$.  There is a complete asymptotic expansion 
\[
\int_{\bbR^n} f(x) e^{\frac{i\langle Ax,x \rangle}{2h}} dx \sim (2 \pi h)^\frac{n}{2} |\det A|^{-\frac{1}{2}}e^{\frac{i\pi}{4} \operatorname{sgn} A} \left( \text{exp} \left(-\frac{ih}{2} b(D)\right) f \right)(0)
\]
where $\operatorname{sgn} A$ is the signature of $A$ and $b(D) = {-}\sum b_{ij} \frac{\partial}{\partial x_i} \frac{\partial}{\partial x_j}$ with $B = A^{-1}$.
\end{lemma}
We are now in a position to prove our main theorem.

\begin{proof} [Proof of {Theorem} \ref{mainthm}]
Let $K_{\rho, h}(x,y)$ be the Schwartz kernel for $\rho(\op_h)$ in some local coordinates.  Recall from \eqref{eqn:altspecmeas} that the spectral measure may be expressed as
\begin{equation}\label{eqn:altspecmeas2}
\mu_{\frac{\alpha}{h}}(\rho) = \int_{\bbT^m} \tr(e^{-\frac{i\alpha\cdot \theta}{h}}\action(\theta)_*\rho(\op_h)) d \theta = \int_{\bbT^m}\int_X e^{-\frac{i \alpha \cdot \theta}{h}} K_{\rho,h}(\action(\theta) x,x) dx d\theta .
\end{equation}
To get an expansion for this measure, we need to work locally.  We take a partition of unity on $X$ given by $\left\{ \varphi_q \right\} \cup \left\{ \varphi_{q,r,s} \right\}$, where the supports of the $\varphi_{q,r,s}$ are coordinate patches on tubular neighborhoods of $X_{r,s}$ and the supports of the $\varphi_q$ are coordinate patches on the complement in $X$ of the union of these tubular neighborhoods .  On $\bbT^m$, we note that $\action(\theta)x = x$ implies that $\theta=0$ or $\theta = \theta_r$ for some $r$.  Let $U_e$ and $U_{e^{i\theta_r}}$ be small neighborhoods of ${e}$ and the $e^{i\theta_r}$'s, respectively, with $V_e \subset U_e$ and $V_{e^{i\theta_r}} \subset U_{e^{i\theta_r}}$ smaller neighborhoods.  Denote by $\psi_0$ and $\left\{ \psi_r \right\}$ functions that are supported in $U_e$ and $U_{e^{i\theta_r}}$ and identically equal to $1$ in $V_e$ and $V_{e^{i\theta_r}}$, respectively.  Using these functions, we may write
\begin{eqnarray}\label{eqn:measureassum}
\mu_{\frac{\alpha}{h}}(\rho) &=& \sum_q \int_{\bbR^m} \int_X e^{-\frac{i \alpha \cdot \theta}{h}} K_{\rho,h} (\action(\theta)x,x) \varphi_q(x) \psi_0 (\theta) dx d\theta \nonumber \\ 
 &&+ \sum_{q,r,s} \int_{\bbR^m}\int_X e^{-\frac{i \alpha \cdot \theta}{h}} K_{\rho,h} (\action(\theta)x,x) \varphi_{q,r,s}(x) \psi_r(\theta) dx d\theta ,
\end{eqnarray}
modulo $O(h^{\infty})$. 

To evaluate the summands {in} \eqref{eqn:measureassum}, we note that on the coordinate patches on which they are defined, the $\bbT^m$-action can be taken to have a special form.  This follows from the Slice Theorem, which we recall here for the sake of completeness {(e.g., \cite[App. B]{ggk})}.  
\begin{thm}[Slice Theorem]
Let $G$ be a compact Lie group acting {smoothly} on a compact manifold $X$. Let {$x$} be a point in $X$ with isotropy group $G_x$. Then a neighborhood of the $x$-orbit $G.x$ is $G$-equivariantly diffeomorphic to the bundle 
$$
G\times_{G_x} D,
$$
where $D$ is a disc {around the origin} in {$N_xG.x$,} the normal space at $x$ to the orbit $G.x$.
\end{thm}
Local triviality implies that a $G$ action as above is locally {equivalent to the $G$ action on $G/G_x\times N_xG.x$, where $G$ acts on the $G/G_x$ factor in the obvious way and acts on $N_xG.x$ via the linearized $G$ action; this equivalence is not canonical.

For $G=\bbT^m$ on $X_0$, there are two types of points to consider. We begin with points where the action is free.  By the Slice Theorem, a neighborhood $U$ of the orbit of such a point admits coordinates $(u,v)=(u_1,\dots, u_{n-m},v_1,\dots, v_m)$  such that
$$
\action(\theta)(u,v)=(u,v+\theta), \, \theta\in \bbR^m.
$$
Let $(\hat{u},\hat{v})$ denote coordinates on the cotangent bundle adapted to the coordinates $(u,v)$ on $U$. The cotangent vector with coordinates $(\hat{u}, \hat{v})$ is $\sum_{l=1}^{n-m}\hat{u}_l du_l+\sum_{l=1}^{m}\hat{v}_l dv_l$.}
Applying Lemma \ref{usual_schwartz_kernel} on such a coordinate patch, a summand on the first line of \eqref{eqn:measureassum} admits an asymptotic expansion in powers of $h$:
$$
  (2\pi h)^{-n} \sum^\infty_{j=0}  h^j \int a_{\rho, j}
    (x,\xi) e^{-\frac{i \alpha \cdot \theta}{h}}e^{\frac{i (\action(\theta)x-x)\cdot \xi}{h}}\varphi_q(x) \psi_0 (\theta) dxd \xi d\theta,
$$
where
$$
  a_{\rho,j} (x,\xi) = \sum_{\ell \leq 2j} b_{j,\ell} (x, \xi)
   \left( \frac{d^\ell \rho}{ds^\ell}  \right) (\fp_0 (x,\xi)).
$$
In the coordinates given by the Slice Theorem, we get
\begin{equation}\label{eqn:first_summand}
(2\pi)^{-n}  \sum_{j=0}^\infty h^{j-n} \int_{\bbR^m} \int_{\bbR^n} \int_{U} a_{\rho,j}(u,v,\hat{u},\hat{v}) e^{\frac{i[\theta \cdot \hat{v}-\alpha \cdot \theta]}{h}} \varphi_q(x) \psi_0 (\theta) dudv d\hat{u}d\hat{v} d\theta .
\end{equation}

Since $\rho(\op_h)$ is $\bbT^m$-invariant, the amplitude $a_{\rho,j}(x,\xi)$ does not depend on $v$ and is only a function of $u$.  Therefore each integral in the sum \eqref{eqn:first_summand} becomes
$$
{\int} a_{\rho,j}(u,\hat{u},\hat{v}) e^{\frac{i(\hat{v}-\alpha)\cdot \theta}{h}} \varphi_q(u,v) \psi_0 (\theta) dud\hat{u}d{v} d\hat{v} d\theta;
$$
note that this integral is finite since the functions $\varphi_q$ and $\psi_0$ have compact support. By changing variables this integral becomes
$$
{\int} a_{\rho,j}(u,\hat{u},\hat{v}+\alpha) e^{\frac{i{\hat{v}}\cdot \theta}{h}} \varphi_q(u,v) \psi_0 (\theta) dud\hat{u}d{v} d\hat{v} d\theta.
$$

{For each $(u,\hat{u},v)$, we are going to apply Lemma \ref{stationary_phase} to the above integral in $(\hat{v},\theta)$.} We can take the matrix $A$ in Lemma \ref{stationary_phase} to be the $2m \times 2m$ matrix given by  
$
A=\begin{bmatrix}
0 & {I} \\

{I}&  0
\end{bmatrix}.
$
We have $|\det A |^{-\frac{1}{2}} = 1$, $\sgn A = 0$, and
$
B=\begin{bmatrix}
0 & I \\
I & 0
\end{bmatrix}
$; since the function $a_{\rho,j}(u, \hat{u}, \hat{v} + \alpha)$ does not depend on $\theta$, we see that applying $b(D)$ to $a_{\rho,j}$ gives $0$.   
Thus each integral in the sum  \eqref{eqn:first_summand} is equal mod $O(h^\infty)$ to 
$$
{(2\pi h)^{m}} {\int} {a_{\rho,j}(u,\hat{u},\alpha)} \varphi_q(u,{v}) dud\hat{u}d{v} ,
$$
which we write as 
$$
{(2\pi h)^{m}} {\int} {a_{\alpha,\rho,j}(u,\hat{u})} \varphi^{\text{av}}_q(u) dud\hat{u} ,
$$
where $a_{\alpha,\rho,j}(u,\hat{u})=a_{\rho,j}(u,\hat{u},\alpha)$ and $\varphi^{\text{av}}_q(u)$ is $\bbT^m$-invariant and independent of $v$. In other words, this summand is equal mod $O(h^{\infty})$ to the integral of $\varphi^{\text{av}}_q(u)a_{\alpha,\rho}$ over $(T^*{U})_\alpha$ with respect to the reduced symplectic form $dud\hat{u}$. By summing over $q$, we see that the contribution of these terms to \eqref{eqn:measureassum} is
$$
{(2\pi h)^{m-n}} \sum_{i=0}^\infty h^i \sum_{\ell\leq 2i} \mu_{i,\ell}{\left(\frac{d^\ell \rho}{dt^\ell}\right)},
$$
with $\mu_{i,\ell}$ defined by $\mu_{i,\ell}=(\fp_\alpha)_*b_{i,\ell}\nu$.  Recall from Lemma \ref{usual_schwartz_kernel} that $b_{0,0}=1$.

Next we treat the contribution of those points that have nontrivial isotropy but are in $X_0$, i.e., points with discrete nontrivial isotropy. Let $x$ be a point that is in the $s$ component of the fixed point set of $e^{i\theta_r}$, so $x \in X_{r,s}$. The Slice Theorem gives a local identification of a tubular neighborhood of $X_{r,s}$ with $D\times \bbT^m/H$, where $D$ is a disc around the origin in the normal bundle to the $\bbT^m$-orbit at $x$ and $H$ is the finite isotropy group. This identification gives coordinates $(u,v)$ in a neighborhood of $x$ such that
{$$
\action(\theta+\theta_r)(u,v)=(A_r u,v+\theta)
$$
for $\theta$ close to zero, where $u=0$ is the fixed point set of $\action(\theta_r)$ and $A_r$ is the linear action of $\action$ on the normal bundle to $u=0$.} More precisely there is $k\geq m$ such that 
$$
A_r(u_1,\dots, u_{n-m})=(u_1,\dots, u_k,\sum_{j=k+1}^{n-m}A_{r,j,k+1}u_{j},\dots,\sum_{j=k+1}^{n-m}A_{r,j,n-m}u_{j})
$$
and we may assume that the matrix $\left(A_{r,i,j}-\delta_{i,j}\right)_{|i,j=k+1,\dots, n-m}$ is invertible. 
Note that $k$ depends on $r$ and $s$.
We denote the corresponding coordinates on the cotangent bundle by $(\hat{u},\hat{v})$ as before. By Lemma \ref{usual_schwartz_kernel}, a summand on the second line of \eqref{eqn:measureassum}
admits an asymptotic expansion in powers of $h$:
$$
  (2\pi h)^{-n} \sum^\infty_{{j}=0}  h^{j} \int a_{\rho,{j}}
    (x,\xi) e^{-\frac{i \alpha \cdot \theta}{h}}e^{\frac{i (\action(\theta)x-x)\cdot \xi}{h}}\varphi_{q,r,s}(x) \psi_r (\theta) dxd \xi d\theta,
$$
where
$$
  a_{\rho,{j}} (x,\xi) = \sum_{\ell \leq 2{j}} b_{{j},\ell} (x, \xi)
    \left( \frac{d^\ell \rho }{ds^\ell}  \right) (\fp_0 (x,\xi)).
$$
Substituting the action above we get
\begin{equation}\label{eqn:second_summand}
(2\pi)^{-n}  \sum_{{j}=0}^\infty h^{{j}-n} \int_{\bbR^m} \int_{\bbR^n} \int_X {a_{\rho,{j}}(u,v,\hat{u},\hat{v})} e^{\frac{i[(A_r-I)u\cdot \hat{u}+(\theta-\theta_r) \cdot \hat{v}-\alpha \cdot \theta]}{h}} \varphi_{q,r,s}(x) \psi_r (\theta) dudv d\hat{u}d\hat{v} d\theta .
\end{equation}
{For fixed variables $u_i,\hat{u}_i, i=1,\dots, k$, and $v$, we wish to apply Lemma \ref{stationary_phase} to the above integral. 

If we make the change of variables
\[
(u_{k+1},\dots,u_{n-m},\hat{u}_{k+1},\dots,\hat{u}_{n-m},\hat{v},\theta) \mapsto (u_{k+1},\dots,u_{n-m},\hat{u}_{k+1},\dots,\hat{u}_{n-m}, \hat{v} - \alpha, \theta - \theta_r)
\]
and denote the new variables by $\hat{x}$, we see that the argument of the exponential function in \eqref{eqn:second_summand} is of the form $\frac{i (\hat{x}^T A \hat{x}-2\alpha\theta_r)}{2h}$, where $A$} is a $2(n-k)$ by $2(n-k)$ matrix {given in block form by
\[
A = \left[
\begin{matrix}
0 & A_r-I & 0 & 0 \\
A_r-I & 0 & 0 & 0 \\
0 & 0 & 0 & I \\
0 & 0 & I & 0
\end{matrix}
\right].
\]
Lemma \ref{stationary_phase} then shows that the expression in \eqref{eqn:second_summand} is equal mod $O(h^\infty)$ to 
\[
(2\pi h)^{-k}e^{-\frac{i\theta_r\cdot \alpha}{h}}|\det(A_r-I)|^{-1} \sum_{j=0}^{\infty}h^j\int \left(\text{exp}\left(-\frac{ih}{2}c(D)\right) a_{\rho,j}\right) (u, v, \hat{u},\hat{v}) \varphi^{\text{av}}_{q,r,s}(u) dudvd\hat{u}.
\]}
Here $c(D)$ denotes the differential operator 
$$
{-}c_{ij}\frac{\partial}{\partial \hat{x}_i}\frac{\partial}{\partial \hat{x}_j}
$$
where $c_{ij}$ are the entries of the matrix $(A_r-I)^{-1}${; also note that the integration is only with respect to the first $k$ components of $u$ and $\hat{u}$}. 

The highest order contribution in \eqref{eqn:second_summand}  is therefore 
\[
(2\pi h)^{-k}e^{-\frac{i\theta_r\cdot \alpha}{h}}|\det(A_r-I)|^{-1} \int{ a_{\rho,0}(u, v, \hat{u},\hat{v}) }\varphi^{\text{av}}_{q,r,s}(u, v) dudvd\hat{u}.
\]
Summing the contributions over $q$ we {see that the highest order contribution to the second line of \eqref{eqn:measureassum} is given by}
$$
(2\pi h)^{-k} e^{-\frac{i\theta_r\cdot \alpha}{h}}|\det(A_r-I)|^{-1} (\fp_\alpha^{r,s})_*\nu_{r,s}.
$$
\end{proof}

\section{Inverse Spectral Results}\label{sec:inverse_spectral_results}
In this section, we apply the asymptotic expansion given in Theorem \ref{mainthm} to get some inverse spectral results.  More precisely, let $(X,g)$ be a Riemannian manifold endowed with an isometric $\bbT^m$ action, and let $\op_h$ be a family of semi-classical pseudo-differential operators on $X$ which commute with $\bbT^m$.  We will extract information about $X$ from knowledge of the spectrum of $\op_h$ together with information about the spectral measure for {\it large} weights of the group representations on the eigenspaces of $\op_h$.  We emphasize that these results only involve information about {\it large} weights of the representations of the group on eigenspaces, whereas our earlier work (e.g., \cite{dgs2}) used information about {\it all} weights of the group representations on the eigenspaces.  Our focus will be on the leading term in the asymptotic expansion given in Theorem \ref{mainthm} for two operators: $\op_h = h^2 \Delta$; and $\op_h = h^2 \Delta +V$, where $V$ is a $\bbT^m$-invariant potential on $X$. 

We begin with a precise description of the spectral data we will use.
{\begin{defn}\label{defn:espec}
Let $(X,g)$ be a Riemannian manifold endowed with an isometric $\bbT^m$ action, and let $\op_h$ be a family of semi-classical pseudo-differential operators on $X$ which commute with $\bbT^m$.  The {\em espec($h$)} of $\op_h$ is the spectrum of $\op_h$, i.e., the eigenvalues of $\op_h$ listed with multiplicities, together with the weights of {the $\bbT^m$-representations on each eigenspace of the form $\frac{\alpha}{h}$}, for all $\alpha \in \bbZ^m$. We say that a quantity is {\em espectrally determined} if it is determined by espec($h$) for some (possibly very small) $h$.
\end{defn}}

It follows from Theorem \ref{mainthm} that all terms in the asymptotic expansion of the spectral measure $\mu_{\frac{\alpha}{h}}$ as defined in \eqref{spectral_measure} are espectrally determined when $\op_h=h^2\Lap$ or $\op_h=h^2\Lap+V$.  

We have seen that in the situation described in Definition \ref{defn:espec}, for generic $\alpha$, a semi-classical pseudo-differential operator $\op_h$ restricted to $\Sm(X)^{\frac{\alpha}{h}}$ gives rise to a symbol $\p_\alpha:(T^*X)_\alpha\rightarrow \bbR$. In \S \ref{sec:classred} we explained that, given a metric on $X$, there is an identification of $(T^*X)_\alpha$ with $T^*(X_0/{\bbT^m})$. This identification carries the induced canonical form on $(T^*X)_\alpha$ coming from symplectic reduction to the form $\omega_0+\nu^\sharp_\alpha$, where $\omega_0$ is the canonical form on $T^*(X_0/{\bbT^m})$ and $\nu^\sharp_\alpha$ was defined in \S\ref{sec:classred}.
Then $\omega_0+\nu^\sharp_\alpha$ defines a volume form on $T^*(X_0/{\bbT^m})$, and therefore a measure which we denote by $\vol_\alpha$.
We will use the following corollary of Theorem \ref{mainthm}. 

\begin{thm}\label{leading_thm}
Let $X$ be a manifold with a $\bbT^m$ action and $\op_h$ a family of  pseudo-differential operators that commute with $\bbT^m$.  {Then for any compactly supported smooth function $\rho$ on $\bbR$, the quantity
$$
\int_{T^*(X_0/{\bbT^m})}\rho\circ \p_\alpha \vol_\alpha
$$
is espectrally determined.} 
\end{thm}

{
\begin{proof}
By Theorem \ref{mainthm}, the highest order term in the asymptotic expansion of the spectral measure $\mu_{\frac{\alpha}{h}}$ is
\[
(2 \pi h)^{m-n} (\p_\alpha)_* \nu (\rho).
\]
We have
\[
(\p_\alpha)_* \nu(\rho) = 
\int_{(T^*X)_\alpha}\rho \circ \p_\alpha \nu = 
\int_{T^*(X_0/{\bbT^m})}\rho\circ \p_\alpha \vol_\alpha,
\]
where the second equality follows from our identification of $(T^*X)_\alpha$ with $T^*(X_0/{\bbT^m})$.
\end{proof} 
}

We will apply Theorem \ref{leading_thm} to prove three types of inverse spectral results.
\begin{itemize}
\item We will show that given a Schr\"odinger operator on $\bbR^{2n}$ with a potential that  {satisifes a convexity-type assumption and} is invariant under the usual $\bbT^n$ action on $\bbR^{2n}$, the potential of the operator is espectrally determined.
\item We will show that a metric on $S^2$ that is $S^1$-invariant, is invariant under the antipodal map, and is suitably convex, is espectrally determined.
\item We will show that given any toric K\"ahler metric on a generic toric orbifold, the {(unlabeled)} moment polytope of the orbifold is espectrally determined. 
\end{itemize}

\subsection{Torus invariant Schr\"odinger operators on $\bbR^{2n}$}

The Schr\"{o}dinger operators that we will consider involve potentials $V$ that satisfy certain conditions.  

\begin{defn}
Let $V$ be a smooth function on $(\bbR_+)^{n}$. We say that  $V$ is \emph{admissible} if
\begin{itemize}
 \item $V$ is proper, i.e., $V(s)\rightarrow \infty$ as $s\rightarrow \infty$;
 \item the partial derivatives of $V$ are positive, i.e., 
 \begin{equation}\label{admissible1}
 \frac{\partial V}{\partial s_i}>0, \, i=1,\dots, n;
 \end{equation}
 \item {and 
\begin{equation}\label{admissible2}
\sum_{i,j} \frac{\partial^2 V}{\partial s_i\partial s_j}c_ic_j \geq 0, \, \text{for all } c\in \bbR^n.
\end{equation}}
\end{itemize}
\end{defn}
The first condition above ensures that the Schr\"odinger operator $h^2\Lap+V$ has discrete spectrum. For example, the function $V(s_1,\dots,s_n)=s_1+\dots+s_n$ is admissible and our theorem will apply to it.

\begin{thm}\label{schrodinger_thm}
Let {$V({s_1, \dots, s_n})$} be an admissible function on $\bbR^n_+$. Consider the Schr\"odinger operator  given by {$-h^2\Lap+V$, where $\Lap= - \sum {\left(\frac{\partial^2}{\partial x_i^2} + \frac{\partial^2}{\partial y_i^2}\right)}$ on $\bbR^{2n}$ {and $V= V(x_1^2+y_1^2, \dots, x_n^2+y_n^2)$}. Then $V$ is determined by the espec of $-h^2\Lap+V$}.
\end{thm}

Note that we are using the same letter $V$ to denote the function defined on $\bbR^n_+$ as well as the corresponding rotationally invariant function $V= V(x_1^2+y_1^2, \dots, x_n^2+y_n^2)$ defined on $\bbR^n$. Although this is an abuse of notation, these two functions determine each other.
We will split the proof of this theorem into two lemmas. The first lemma involves making Theorem \ref{leading_thm} explicit in a very simple context. Namely, let $X=\bbR^{2n}=\bbR^2\times \cdots \times \bbR^2$ be endowed with the usual $\bbT^n=S^1\times \cdots \times S^1$ action where $S^1$ acts on $\bbR^2$ by rotations. We will use polar coordinates $(r_i, \theta_i)$ on the $i$th $\bbR^2$ factor, for $i=1, \dots, n$. In these coordinates the $\bbT^n$ action is given by
$$
(e^{i\xi_1},\dots,e^{i\xi_n})\cdot (r_1e^{i\theta_1},\dots,r_ne^{i\theta_n})=(r_1e^{i(\theta_1+\xi_1)},\dots,r_ne^{i(\theta_n+\xi_n)}),
$$
with $(e^{i\xi_1},\dots,e^{i\xi_n})\in \bbT^n$.

\begin{lemma}\label{schrodinger1_lemma}
Let $V$ be a scalar function on ${\bbR^{n}_+}$. Consider the Schr\"odinger operator given by {$-h^2\Lap+V$, where $\Lap= - \sum {\left(\frac{\partial^2}{\partial x_i^2} + \frac{\partial^2}{\partial y_i^2}\right)}$ on $\bbR^{2n}$} and $V= V(x_1^2+y_1^2, \dots, x_n^2+y_n^2)$ is a $\bbT^n$-invariant potential. Given $\alpha\in \bbR^n$, the function
$$
m(\alpha)=\min \left(V+\sum_{i=1}^n \frac{\alpha_i^2}{r_i^2}\right)
$$ 
is espectrally determined.
\end{lemma}
\begin{proof}
Our manifold $X=\bbR^{2n}$ carries the usual flat metric. A $\bbT^n$-invariant Schr\"odinger operator for this metric on $X$ is {$-h^2 \Lap+V$}, which is a pseudo-differential operator with symbol $\p:T^*\bbR^{2n}\rightarrow \bbR$ given by
$$
\p({\bf x},{\bf v})=||{\bf v}||^2+V({\bf x}),\, \ \ \  ({\bf x},{\bf v})\in T^*\bbR^{2n}.
$$

We can make the moment map and the construction of the relevant forms as in \S \ref{sec:classred} completely explicit in this case. We start by noticing that the $\bbT^n$ action is free on $X_0=(\bbR^2\setminus\{0\})\times \cdots \times (\bbR^2\setminus\{0\})$ and that
$$
X_0/\bbT^n=\bbR_+^n.
$$
The vector field induced by the $\bbT^n$ action is
$$
V_\xi=\sum_{i=1}^n \xi_i \frac{\partial}{\partial \theta_i},
$$
for any $\xi\in \bbR^n$. On $T^*\bbR^{2n}$ we pick coordinates $$(r_1, \dots, r_n,\theta_1,\dots,\theta_n, u_1, \dots, u_n,v_1, \dots, v_n)$$ which correspond to the cotangent vector $\sum_{i=1}^n u_idr_i+v_id\theta_i$. In these coordinates the moment map $\Phi:T^*\bbR^{2n}\rightarrow (\bbR^n)^*$ induced by the $\bbT^n$ action on $T^*\bbR^{2n}$ is given by
$$
\langle \Phi(r,\theta, u,v), \xi \rangle = \sum_{i=1}^n (u_idr_i+v_id\theta_i)(V_\xi(r,\theta)) = \sum_{i=1}^nv_i\xi_i.
$$
Therefore, given $\alpha\in (\bbR^n)^*\simeq \bbR^n$,  $\Phi^{-1}(\alpha)$ is simply
$$
\Phi^{-1}(\alpha)=\{(r,\theta,u,v):v=\alpha\}.
$$

To construct the relevant forms, we begin by observing that there is a $1$-form $\mu_\alpha$ defined uniquely on $X_0$ by
\begin{enumerate}
\item $\mu_\alpha(V_\xi(x))=\langle \alpha,\xi\rangle$ for any $x\in X_0$ and $\xi\in \bbR^n$; and
\item $\mu_\alpha=0$ on the orthogonal complement of the tangent space to $\bbT^n$-orbits.
\end{enumerate}
Since $\{( \frac{\partial}{\partial r_i},\frac{\partial}{\partial \theta_i}),\, i=1,\dots, n\}$ is an orthogonal basis for the tangent space of $X_0$, these two conditions imply 
$$\mu_\alpha=\sum_{i=1}^n \alpha_i d\theta_i$$
for $\alpha=(\alpha_1,\dots,\alpha_n)$. Thus $d\mu_\alpha$ is zero and the form $\nu_\alpha^\sharp$ which is induced from $d\mu_\alpha$ on $T^*(X_0/\bbT^n)$ using the pullback from $X_0/\bbT^n$ to $T^*(X_0/\bbT^n)$ is also zero. We conclude that the measure on $T^*(X_0/\bbT^n)$ induced by its identification with $T^*X_\alpha$ is the standard measure on $T^*(X_0/\bbT^n)$ coming from its canonical form. 

The symbol of our Schr\"odinger operator induces a symbol $\p_\alpha:(T^*X)_\alpha\rightarrow \bbR$ which is given by
$$
\p_\alpha(r,\theta,u,\alpha)=\sum_{i=1}^n (u_i^2||dr_i||^2+\alpha_i^2||d\theta_i||^2)+V(r_1^2,\dots,r_n^2).
$$
{Recalling that} $(r_i,\theta_i)$ are polar coordinates for the $(x_i,y_i)$ coordinates{, we} have
$$
dr_i=\frac{x_idx_i+y_idy_i}{r_i}, \, d\theta_i=\frac{-y_idx_i+x_idy_i}{r_i^2};
$$
in particular, $||dr_i||^2=1$ and $||d\theta_i||^2=\frac{1}{r_i^2}$ so that
$$
\p_\alpha(r,\theta,u,\alpha)=\sum_{i=1}^n \left(u_i^2+\frac{\alpha_i^2}{r_i^2}\right)+V(r_1^2,\dots,r_n^2). 
$$
It follows from Theorem \ref{leading_thm} {that the integral}
$$
\int_{\bbR_+^n\times \bbR^n}\rho\left(\sum_{i=1}^n \left(u_i^2+\frac{\alpha_i^2}{r_i^2}\right)+V(r_1^2,\dots,r_n^2)\right)du_1\wedge \cdots \wedge du_n \wedge dr_1 \wedge \cdots \wedge dr_n
$$
is espectrally determined for any compactly supported $\rho$. In particular, the minimum $m=m(\alpha_1,\dots,\alpha_n)$ of 
$$
\sum_{i=1}^n \left(u_i^2+\frac{\alpha_i^2}{r_i^2}\right)+V(r_1^2,\dots,r_n^2)
$$
is espectrally determined. This is because the integral above will be zero for all functions $\rho$ which have support in $]-\infty,m[$. But $m(\alpha_1,\dots,\alpha_n)$ is actually the minimum of
$$
\sum_{i=1}^n \frac{\alpha_i^2}{r_i^2}+V(r_1^2,\dots,r_n^2),
$$
and the lemma follows. \end{proof}

Note that Lemma \ref{schrodinger1_lemma} does not require $V$ to be admissible, but the next lemma does. In fact, Theorem \ref{schrodinger_thm} will hold for any $V$ for which Lemma \ref{schrodinger2_lemma} is satisfied, and will follow from Lemma \ref{schrodinger1_lemma} if we can show that $m(\alpha)$ determines $V$.
\begin{lemma}\label{schrodinger2_lemma}
Let $V$ be {an admissible} function on $\bbR_+^n$. Let $m(\alpha_1,\dots,\alpha_n)$ denote the minimum value of
\[
\sum_{i=1}^n \frac{\alpha_i^2}{r_i^2}+V(r_1^2,\dots,r_n^2).
\]
The function $m$ determines $V$.
\end{lemma}
\begin{proof}
Let $s_i=r_i^2$ . The minimum of 
$$
\sum_{i=1}^n \frac{\alpha_i^2}{s_i}+V(s_1,\dots,s_n)
$$
occurs at a point where we have
$$
\frac{\partial V}{\partial s_i}=\frac{\alpha_i^2}{s_i^2}\text{\ \ or equivalently, \ \ } s_i^2\frac{\partial V}{\partial s_i}=\alpha_i^2,\text{\, for all\ } i=1,\dots, n.
$$
Set {$t_i=\frac{1}{s_i}$} and let
$$
F({t_1, \dots, t_n})=-V\left(\frac{1}{{t_1}}, \dots,\frac{1}{{t_n}}\right).
$$
Note that $m(\alpha_1,\dots,\alpha_n)$ is the minimum value of $\sum_{i=1}^n \alpha_i^2 t_i - F(t_1, \dots, t_n)$.  For any $i=1,\dots, n$, at a point where the minimum is achieved we have
$$
\frac{\partial F}{\partial {t}_i}=-\frac{\partial V}{\partial s_i}\frac{\partial s_i}{\partial {t}_i}=\frac{1}{({t}_i)^2}\frac{\partial V}{\partial s_i}=s_i^2\frac{\partial V}{\partial s_i}=\alpha_i^2.
$$
Moreover,
$$
\frac{\partial^2 F}{\partial t_i\partial t_j}=-\frac{\partial^2 V}{\partial s_i\partial s_j}\frac{1}{t_i^2t_j^2}-2\delta_{ij}\frac{\partial V}{\partial s_i}\frac{1}{t_i^3}.
$$

If $V$ is admissible, condition \eqref{admissible1} ensures that the Legendre transform 
$$
\frac{\partial F}{\partial {t}}:\bbR_+^n \rightarrow  \bbR^n
$$
maps $\bbR_+^n$ into $\bbR_+^n$. {By \eqref{admissible2}, $-F$ is strictly convex and therefore $\frac{\partial F}{\partial t}$ is a diffeomorphism onto an open subset of $\bbR_+^n$ (see \cite[\S A1.3]{GuMM}). We claim that this open set is, in fact, all of $\bbR_+^n$.  To see this, consider
$$
\sum_{i=1}^n\alpha_i^2t_i-F(t)=V\left(\frac{1}{t_1}, \dots,\frac{1}{t_n}\right)+\sum_{i=1}^n\alpha_i^2t_i.
$$
The expression on the right tends to infinity as $t$ tends to infinity. The properness of $V$ implies that this expression also tends to infinity as $t$ approaches the boundary of $\bbR_+^n$. The left side therefore must have a minimum value at some point $t_0$, and 
$$
\frac{\partial F}{\partial t}(t_0)=(\alpha_1^2,\dots, \alpha_n^2).
$$

Now let $G: \bbR^n_+ \rightarrow \bbR$ be the function $G(a)=t\cdot a -F(t)$ at $a=\frac{\partial F}{\partial t}({t})$. Thus {for $a_i=\alpha_i^2$, $i=1,\dots, n$, we have $G(a)=m(\alpha)$}. Then $G(a)$ is espectrally determined and hence so is its Legendre transform $\frac{\partial G}{\partial a}$. But  $\frac{\partial G}{\partial a}$ is the inverse of $\frac{\partial F}{\partial t}$ at $t=\frac{\partial G}{\partial a}(a)$, so $F$ is espectrally determined. Given the relationship between $F$ and $V$, this shows that $V$ is espectrally determined.}
\end{proof}

\subsection{Laplace operators on toric orbifolds}
\subsubsection{Background on toric orbifolds}
We briefly recall some basic facts about toric orbifolds and K\"ahler toric metrics on them. For more details, see \cite{a2}. 
\begin{defn}
Let $(X,\omega)$ be a compact connected symplectic orbifold of real dimension $2n$. Then $(X,\omega)$ is said to be \emph{toric} if it admits an effective Hamiltonian $\bbT^n$ action, where $\bbT^n$ is the real torus of dimension $n$.
\end{defn}

{The existence of an effective Hamiltonian action implies that there is a \emph{moment map} from $X$ to $\bbR^n$. It follows from} the convexity theorem of Atiyah and Guillemin and Sternberg that the image of the moment map of a toric orbifold is convex; it is the convex hull of the fixed points of the $\bbT^n$ action. In fact this convex hull is always a  \emph{rational simple} polytope.
\begin{defn}\label{defn:rat_simple}
A convex polytope $P$ in $\mathbb{R}^n$ is \emph{rational simple} if
\begin{enumerate}
\item there are $n$ facets meeting at each vertex;
\item for every facet of $P$, a primitive outward normal can be chosen in $\bbZ^n$;
\item for every vertex of $P$, the outward normals corresponding to the facets meeting at that vertex form a basis for $\bbQ^n$.
\end{enumerate}
Note that a facet is a face in $P$ of codimension $1$.
\end{defn}
Given a symplectic toric orbifold $(X,\omega)$ whose moment map image is a rational simple polytope $P$, the action of $\bbT^n$ is free exactly on the pre-image via the moment map of the interior of $P$. The pre-image via the moment map of {$\interior(P)$} is an open dense subset of $X$, and $X$ admits so-called \emph{action-angle coordinates} on that open set. 
\begin{defn}
Let $(X,\omega)$ be a toric symplectic orbifold. On the pre-image of $\interior(P)$, we define \emph{action-angle coordinates} on $X$ by 
$$(x_1,\dots,x_n,\theta_1, \dots, \theta_n),$$ where $(x_1,\dots,x_n)$ are the coordinates of the moment map and have image in $P$ and  $(\theta_1, \dots, \theta_n)$ are angle coordinates on the torus $\bbT^n$.
\end{defn} 
\noindent Note that if $(x_1,\dots,x_n,\theta_1, \dots, \theta_n)$ are action-angle coordinates on $(X,\omega)$, then
$$
\omega=\sum_{i=1}^{n} dx_i\wedge d\theta_i.
$$

Lerman and Tolman \cite{lt} showed that for every facet of $P$, there is an integer $m$ such that the structure group of any point in the pre-image of the facet via the moment map has structure group $\bbZ/m\bbZ$. The collection of integers attached to the facets are called the \emph{labels} of the moment polytope. 
\begin{defn}
A  \emph{labeled} polytope is a rational simple polytope with an integer associated to each facet.
\end{defn}
Given a labeled polytope $P$ with facets $F_1, \dots, F_d$, we will denote by ${\n}^0_i$ the unique inward-pointing normal vector to $F_i$ which is a primitive element of the lattice $\bbZ^n\subset \bbR^n$. It is always possible to characterize a rational simple polytope $P$ as
$$
P=\{x\in \bbR^n: x\cdot m_i\n^0_i \geq\lambda_i,\, i=1,\dots, d\},
$$
where $m_i$ is the integer label of the facet $F_i$ and $\lambda_i$ is a constant. We set $\n_i=m_i\n^0_i$ for $i=1,\dots,d$, and
$$
l_i(x)=x\cdot \n_i -\lambda_i, \, i=1,\dots, d.
$$
Lerman and Tolman generalized the well-known classification theorem of Delzant for toric manifolds to the orbifold setting.
\begin{thm}\cite{lt}\label{lt}
To every toric symplectic orbifold we can associate a labeled polytope. Conversely, given any labeled polytope there is a toric symplectic orbifold associated to it. Two toric symplectic orbifolds are equivariantly symplectomorphic if and only if their labeled polytopes are isomorphic.
\end{thm}

It turns out that toric symplectic manifolds always have a K\"ahler metric (see \cite{g}). The same is true for toric symplectic orbifolds. Given a labeled polytope P we associate to it the following function on $\interior(P)$:
$$
\s_P(x)=\sum_{i=1}^d l_i(x)\log(l_i(x))-l_i(x).
$$
This function is used in Abreu's description of all K\"{a}hler metrics on a toric symplectic orbifold.
\begin{thm}\cite{a3}
Given a labeled polytope $P$, K\"ahler metrics on the corresponding toric orbifold $X_P$ are in one-to-one correspondence with potentials $\s:\interior(P)\rightarrow \bbR$ satisfying
\begin{itemize}
\item the Hessian $\hess(\s)$ is positive definite on $P$;
\item the function $\s-\s_P$ extends smoothly to a neighborhood of the boundary of $P$; and
\item $\det(\hess(\s))\prod_{i=1}^d l_i$ extends smoothly to a neighborhood of the boundary of $P$. 
\end{itemize}
Given such a potential, the Riemannian metric associated to the K\"ahler structure can be given explicitly in action-angle coordinates over the pre-image via the moment map of $P$. More precisely, the metric is 
\begin{equation} \label{toric_metric_matrix}
\begin{bmatrix}
\phantom{-}\hess(\s) & \vdots & 0\  \\
\hdotsfor{3} \\
\phantom{-}0 & \vdots & \hess^{-1}(\s)
\end{bmatrix}
\end{equation}
\end{thm}
The function $\s$ is called the \emph{symplectic potential} for the K\"ahler metric.

\begin{lemma}\label{boundary_behavior}
The matrix $\hess^{-1}(\s)$ has the following properties:
\begin{enumerate}
\item $\hess^{-1}(\s)$ extends smoothly to a neighborhood of the boundary of $P$; 
\item $\hess^{-1}(\s)\n_i=0$ on $F_i$, for $i=1, \dots, d$; and
\item on $F_i$, we have
$$
{\frac{\n_i^t\hess^{-1}(\s)\n_i}{l_i}}=1,
$$
for $i=1,\dots, d$.
\end{enumerate}
\end{lemma}
\begin{proof}
The first two statements are well known (e.g., \cite{a3}). We will prove the third statement. In fact, a slight modification of this proof would also show the first two statements in the lemma.

We will assume, without loss of generality, that $i=1$. We begin with the case $\s=\s_P$.  
The $(a,b)$-entry of the matrix $\hess(\s_P)$ is given by 
$$
(\hess(\s_P))_{ab}=\sum_{i=1}^d\frac{\n_i^a \n_i^b}{l_i},
$$
{where $\n_i^k$ denotes the $k$th component of the vector $\n_i$.}
By using Cramer's rule we see that for $s=1,\dots,{n}$, the $s$ component of $\hess^{-1}(\s_P)\n_1$ is given by
$$
\frac{\det \left(\sum_{i=1}^d\frac{\n_i^1 \n_i}{l_i},\dots,\sum_{i=1}^d\frac{ \n_i^{s-1}  \n_i}{l_i},\n_1,\sum_{i=1}^d\frac{ \n_i^{s+1}  \n_i}{l_i},\dots,\sum_{i=1}^d\frac{ \n_i^n  \n_i}{l_i}\right)}{\det \left(\sum_{i=1}^d\frac{ \n_i^1  \n_i}{l_i},\dots,\sum_{i=1}^d\frac{ \n_i^n  \n_i}{l_i}\right)}.
$$
Expanding this out we see that $(\hess^{-1}(\s_p)\n_1)^s$ is given by
$$
\frac{\displaystyle\sum_{\substack{(i_1,\dots, i_{s-1},i_{s+1},\dots, i_n) \\ \text{all distinct},\ne 1}}\frac{\n_{i_1}^1\dots  \n_{i_{s-1}}^{s-1} \n_{i_{s+1}}^{s+1}\dots  \n_{i_{n}}^{n}}{l_{i_1}\dots l_{i_{s-1}}l_{i_{s+1}}\dots l_{i_n}}\det(\n_{i_1},\dots,\n_{i_{s-1}},\n_1,\n_{i_{s+1}},\dots, \n_{i_n})}{\displaystyle\sum_{\substack{(i_1,\dots, i_n)\\ \text{all distinct}}}\frac{\n_{i_1}^1\dots \n_{i_n}^n}{l_{i_1}\dots l_{i_{n}}} \det(\n_{i_1},\dots, \n_{i_n})}.
$$
This implies that $\n_1^t\hess^{-1}(\s_p)\n_1$ is given by
$$
\frac{\displaystyle\sum_{s=1}^n\displaystyle\sum_{\substack{(i_1,\dots, i_{s-1},i_{s+1},\dots, i_n) \\ \text{all distinct},\ne 1}}\frac{\n_{i_1}^1\dots \n_{i_{s-1}}^{s-1}\n_1^s\n_{i_{s+1}}^{s+1}\dots \n_{i_{n}}^{n}}{l_{i_1}\dots l_{i_{s-1}}l_{i_{s+1}}\dots l_{i_n}}\det(\n_{i_1},\dots,\n_{i_{s-1}},\n_1,\n_{i_{s+1}},\dots, \n_{i_n})}{\displaystyle\sum_{\substack{(i_1,\dots, i_n)\\ \text{all distinct}}}\frac{\n_{i_1}^1\dots \n_{i_n}^n}{l_{i_1}\dots l_{i_{n}}} \det(\n_{i_1},\dots, \n_{i_n})};
$$
we will denote the numerator of this expression by $\num$.
The denominator can be written as the sum of two terms, namely
$$
\displaystyle\sum_{s=1}^n\displaystyle\sum_{\substack{(i_1,\dots,i_{s-1},i_{s+1},\dots, i_n)\\ \text{all distinct}{, \ne 1}}}\frac{\n_{i_1}^1\dots \n_{i_{s-1}}^{s-1}\n_1^s \n_{i_{s+1}}^{s+1} \n_{i_n}^n}{l_{i_1}\dots l_{i_{s-1}}l_1l_{i_{s+1}}\dots l_{i_n}} \det(\n_{i_1},\dots,\n_{i_{s-1}},\n_1,\n_{i_{s+1}},\dots, \n_{i_n}),
$$
which equals $\num/l_1$, and
$$
\mathcal{M}:= \displaystyle\sum_{\substack{(i_1,\dots, i_n)\\ \text{all distinct}\\ \ne 1}}\frac{\n_{i_1}^1\dots \n_{i_n}^n}{l_{i_1}\dots l_{i_{n}}} \det(\n_{i_1},\dots, \n_{i_n}).
$$
Therefore $\n_1^t\hess^{-1}(\s_p)\n_1=
{\frac{l_1}{1+\frac{l_1\mathcal{M}}{\num}}}.$
On $F_1$,
$$\frac{\n_1^t\hess^{-1}(\s_p)\n_1}{l_1}=1.$$

Next we examine the general case
$\hess(\s)=\hess(\s_P)+G$,
where $G$ is smooth in a neighborhood of the closure of $P$. We can write this as 
$$
\hess(\s)=\hess(\s_P)(1+\hess^{-1}(\s_P)G)
$$
so that
$$
{\hess^{-1}(\s)=(1+\hess^{-1}(\s_P)G)^{-1}\hess^{-1}(\s_P)=\left(\sum_{k=0}^\infty (-1)^k(\hess^{-1}(\s_P)G)^k\right)\hess^{-1}(\s_P)}
$$
when the sum is convergent.  Thus
$$
\frac{\n_1^t\hess^{-1}(\s)\n_1}{l_1}=\frac{\n_1^t\hess^{-1}(\s_P)\n_1}{l_1}+\left(\sum_{k=1}^\infty {(-1)^k \n_1^t (}\hess^{-1}(\s_P)G)^k\right)\frac{\hess^{-1}(\s_P)\n_1}{l_1}.
$$
We have seen that $\frac{\hess^{-1}(\s_P)\n_1}{l_1}$ is smooth. We restrict our attention to $F_1$.  By property (ii), ${\hess^{-1}(\s_P)\n_1}=0$, implying that for any $k\geq 1$ we have 
$$
{(-1)^k}\n_1^t(\hess^{-1}(\s_P)G)^k={(-1)^k}\n_1^t(\hess^{-1}(\s_P)G)(\hess^{-1}(\s_P)G)^{k-1}=0.
$$
It follows that on $F_1$,
$$
\frac{\n_1^t\hess^{-1}(\s)\n_1}{l_1}=\frac{\n_1^t\hess^{-1}(\s_P)\n_1}{l_1}=1.
$$
\end{proof}

The leading term in the asymptotic expansion of $\mu_{\frac{\alpha}{h}}$ can be made very explicit in the case of a toric symplectic orbifold. {Note that contributions to the leading term in Theorem \ref{mainthm} come only from points where the action is free; as mentioned above, the action of $\bbT^n$ is free exactly on the pre-image via the moment map of the interior of the rational simple polytope $P$ corresponding to the toric orbifold.  Every point in the pre-image of $\operatorname{int}(P)$ is regular, so the coordinate-based arguments used in the proof of Theorem \ref{mainthm} apply directly to this case.  In particular, because} we have action-angle coordinates on a dense subset of the toric orbifold, we can essentially treat the toric orbifold as $\bbR^{2n}$. The following lemma can be viewed as a toric version of Theorem \ref{leading_thm}.
\begin{lemma}\label{toric_laplacian1_lemma}
Let $X$ be a toric orbifold endowed with a toric K\"ahler metric with symplectic potential $\s$, and let $\op_h=h^2\Lap$ be the semiclassical Laplace operator associated to this metric. Given a generic $\alpha\in \bbR^n$ and a compactly supported function $\rho \in C_0^{\infty}(\bbR)$, the quantity
$$
\int_{P\times \bbR^n}\rho(\alpha^t\hess \s(x) \alpha +u^t \hess^{-1} \s (x)u)dx_1\wedge\cdots\wedge dx_n\wedge du_1\wedge\cdots\wedge du_n
$$
is espectrally determined.
\end{lemma}
\begin{proof}
Let $(x_1,\dots,x_n,\theta_1,\dots,\theta_n)$ be action-angle coordinates for $X$. Given a generic $\alpha\in \bbR^n$ , we claim:
\begin{itemize}
\item if $\Phi$ is the moment map for the induced $\bbT^n$ action on $T^*X$, then
$$
\Phi^{-1}(\alpha)=\{(x,\theta,u,v): x\in \interior(P),\, \theta,u,v\in \bbR^n, v=\alpha\}; \text{and}
$$
\item the form $\mu_\alpha$ defined in \S \ref{sec:classred} is closed and the measure induced on $T^*(X_0/\bbT^n)= T^*\operatorname{int}(P)$ is just the Lebesgue measure $dx_1\wedge\cdots\wedge dx_n\wedge du_1\wedge\cdots\wedge du_n$.
\end{itemize}

The proofs of these claims are essentially the same as the proofs of the corresponding results in $\bbR^{2n}$, which are contained in the proof of Lemma \ref{schrodinger1_lemma}. {In fact, one may employ those arguments almost verbatim, with a few small changes that we now mention.

We saw in \S \ref{sec:classred} that if $\alpha$ is a regular value for $\Phi$, then $\Phi^{-1}(\alpha)$ is contained in $T^*X_0$. Therefore we can use action-angle coordinates $(x_1,\dots, x_n,\theta_1,\dots,\theta_n)$ on $X_0$.  We have
$$
X_0/\bbT^n\simeq \interior (P).
$$
On $T^*X_0$ we pick coordinates 
$$(x_1, \dots, x_n,\theta_1,\dots,\theta_n, u_1, \dots , u_n,v_1, \dots, v_n)$$
 which correspond to the cotangent vector $\sum_{i=1}^n u_idx_i+v_id\theta_i$. Given $\alpha\in  \bbR^n$ a regular value for $\Phi$,  $\Phi^{-1}(\alpha)$ is simply
$$
\Phi^{-1}(\alpha)=\{(x,\theta,u,v):x\in \interior(P), \theta,u,v \in \bbR^n,v=\alpha\}.
$$

For the $1$-form $\mu_{\alpha}$, we note that the tangent space to a $\bbT^n$-orbit is spanned by the $\frac{\partial}{\partial \theta_i}, \, i=1,\dots, n$. The orthogonal complement is  spanned by the $\frac{\partial}{\partial x_i}, \, i=1,\dots, n$ because the toric metric is of the form \eqref{toric_metric_matrix}. We have 
$$\mu_\alpha=\sum_{i=1}^n \alpha_i d\theta_i,$$
and the rest of the argument showing that $\mu_\alpha$ is closed and we have Lebesgue measure on $T^*P$ follows as in the proof of Lemma \ref{schrodinger1_lemma}.}

The symbol of the Laplace operator on $X$ is
$$
\p(x,\theta,u,v)=||(u,v)||^2.
$$
The matrix of the metric defined by the symplectic potential $\s$ on $X$ is given as in \eqref{toric_metric_matrix} by
$$
\begin{bmatrix}
\phantom{-}\hess(\s) & \vdots & 0\  \\
\hdotsfor{3} \\
\phantom{-}0 & \vdots & \hess^{-1}(\s)
\end{bmatrix}.
$$
So the matrix of the corresponding metric on the cotangent bundle is
$$
\begin{bmatrix}
\phantom{-} \hess^{-1}(\s) & \vdots & 0\  \\
\hdotsfor{3} \\
\phantom{-}0 & \vdots & \hess(\s)
\end{bmatrix},
$$
which implies 
$$
\p(x,\theta,u,v)=||(u,v)||^2=u^t\hess^{-1}(\s)u+v^t\hess(\s)v.
$$
When we restrict to $(T^*X)_{\alpha}$, we get
$$
\p_\alpha(x,u)=u^t\hess^{-1}(\s)u+\alpha^t\hess(\s)\alpha.
$$
The lemma now follows from Theorem \ref{leading_thm}.
\end{proof}

\subsubsection{Inverse spectral results for toric Laplacians on $S^2$}
Let $S^1$ act on $S^2$ by rotation around the vertical axis. This action is toric for the Fubini-Study metric on $S^2${, and by normalizing we can assume that the moment polytope is the interval $[-1,1]$}.  Consider action-angle coordinates $(x,\theta)$ on $S^2$ with $x\in[-1,1]$. Any $S^1$-invariant metric on $S^2$ is of the form
$$
\ddot \s dx\otimes dx+\frac{d\theta\otimes d\theta}{\ddot \s},
$$
where $\ddot \s>0$ is such that 
$$
\ddot \s-\frac{1}{1-x^2}
$$
is smooth. We will prove the following
\begin{thm}
Suppose that {$\ddot\s$ is even and convex. Then the $S^1$-invariant metric is espectrally determined.}
\end{thm}
\begin{proof}
It follows from Lemma \ref{toric_laplacian1_lemma} with $n=1$ that given a metric on $S^2$ of the form
$$
\ddot \s dx\otimes dx+\frac{d\theta\otimes d\theta}{\ddot \s},
$$
the quantity
$$
\int_{[-1,1]\times \bbR}\rho\left(\alpha^2\ddot \s+\frac{u^2}{\ddot \s}\right)dx\wedge du
$$
is espectrally determined. Note that this quantity is 
\begin{equation}\label{eqn:s1quantity}
\alpha\int_{[-1,1]\times \bbR}\rho\left(\alpha^2\left(\ddot \s+\frac{v^2}{\ddot \s}\right)\right)dx\wedge dv ,
\end{equation}
so $\alpha$ will not play a role; we set $\alpha=1$. We next make a change of variable, setting
$$
t=\ddot \s+\frac{{v}^2}{\ddot \s}\text{, or equivalently, } {v}^2=\ddot \s(t-\ddot \s).
$$
In the $(x,t)$ variables, \eqref{eqn:s1quantity} becomes
\begin{equation}\label{eqn:s1xt}
\int_{-1}^1\int_{\ddot \s(x)}^{+\infty} \rho(t) \frac{\sqrt {\ddot \s}}{2\sqrt{t-\ddot \s}} dtdx.
\end{equation}

Our hypotheses that $\ddot \s$ is even and convex imply that $\ddot \s$ has a minimum at $0$ and is increasing on $[0,+\infty)$. Let $c$ be the minimum value of $\ddot \s$.  Note that{, since $\ddot \s >0$,} $c$ is also the minimum of $\ddot \s+\frac{u^2}{\ddot \s}.$
If the support of $\rho$ is contained in $(-\infty,c)$, the integral
$$
\int_{[-1,1]\times \bbR}\rho\left(\ddot \s+\frac{u^2}{\ddot \s}\right)dx\wedge du
$$
vanishes; since this integral is espectrally determined, we see that $c$ is espectrally determined.

Set $w=\ddot \s-c$, and $\tau=t-c$. The integral \eqref{eqn:s1xt} becomes
$$
\int_{-1}^1\int_{w(x)}^{+\infty} \rho(\tau+c) \frac{\sqrt {w+c}}{2\sqrt{\tau-w}} d\tau dx.
$$
By switching the order of integration and renaming $\rho(\tau+c)$, the above integral becomes 
$$
\int_{0}^{+\infty}\int_{0}^{{w^{-1}(\tau)}} \rho(\tau) \frac{\sqrt {w+c}}{\sqrt{\tau-w}} dxd\tau.
$$
We conclude that the function
$$
t\mapsto \int_{0}^{{w^{-1}(t)}}  \frac{\sqrt {w+c}}{\sqrt{t-w}} dx
$$
is espectrally determined. Now set $y=w(x)$. {Then} 
\begin{equation}\label{eqn:Abeltransform}
t\mapsto \int_{0}^{t}  \frac{\sqrt {y+c}}{\sqrt{t-y}} \frac{dw^{-1}}{dy}(y)dy
\end{equation}
is espectrally determined. 

{The function in \eqref{eqn:Abeltransform} can be viewed as the Abel transform of another function, as we now explain. Recall that the fractional integration operation of Abel is defined as
\[
J^a g(s) = \frac{1}{\Gamma(a)} \int_0^s (s-\nu)^{a-1} g(\nu) d\nu,
\]
for $a>0$.  {F}ractional integration of order $\frac{1}{2}$ applied to the function
\begin{equation}\label{eqn:abel}
y\mapsto \sqrt {y+c}\frac{dw^{-1}}{dy}(y)
\end{equation}
gives \eqref{eqn:Abeltransform}, up to a constant factor of $\Gamma(\frac{1}{2})=\sqrt{\pi}$.   We claim that the Abel transform of a function determines the function.  In particular, we have $J^a J^b = J^{a+b}$.  So if we let $a=1-b$, we have 
\[
J^{1-b} J^{b} g(s) = J^1 g(s) = \int_0^s g(\nu) d\nu .
\]
Taking the derivative of the right side, we recover $g(s)$.  Applying this to our situation, we see that the function in \eqref{eqn:abel} is espectrally determined.  Thus $w$, and in fact $\ddot \s$, are also espectrally determined, which implies that the metric is espectrally determined.}
\end{proof}

\subsubsection{Toric Laplacian on toric orbifolds}
In \cite{dgs2} the authors proved that the equivariant spectrum of a generic toric orbifold determines the toric orbifold up to equivariant symplectomorphism. This is done by proving that, generically, the equivariant spectrum determines the labeled polytope of the toric orbifold{, up to two choices and up to translation,} and then using Theorem \ref{lt}. 
{Our genericity assumptions are that the moment polytope has no parallel facets and no subpolytopes, where a ``subpolytope'' is a convex polytope that can be formed from a proper nonempty subset of the facet normals and associated volumes of our moment polytope.  See \cite[Lemma 5.8]{dgs2} for more details.}

In this section we will show that{, under these same genericity assumptions,} the {unlabeled} moment polytope is espectrally determined. This means that we do not need the full equivariant spectrum to determine the moment polytope, but only ``large" weights for the induced action on eigenspaces. 
\begin{thm}
Let $(X,\omega)$ be a generic toric orbifold endowed with any toric K\"ahler metric, and let $\op_h = h^2 \Delta$. The {unlabeled} moment polytope of $X$ is espectrally determined{, up to two choices and up to translation.}
\end{thm}
\begin{proof}
Let $P$ be the moment polytope of $X$, and let $g$ be the symplectic potential for the toric K\"ahler metric. We write $H$ for $\hess(g)$. Given any $\rho \in C_0^\infty(\bbR)$ and a generic $\alpha\in \bbR^n$, Lemma \ref{toric_laplacian1_lemma} implies that
$$
\int_{P\times \bbR^n}\rho(\alpha^tH(x) \alpha +u^tH^{-1}(x)u)dx_1\wedge\cdots\wedge dx_n\wedge du_1\wedge\cdots\wedge du_n
$$
is espectrally determined, where $(x_1,\dots,x_n)$ are coordinates on $P$. Fix $x\in P$. There is an orthogonal matrix $S$ which diagonalizes $H^{-1}$. Let $\lambda_1(x),\dots, \lambda_n(x)$ be the eigenvalues of $H^{-1}(x)$ and let $\Lambda(x)$ be the diagonal $n \times n$ matrix with positive entries $\lambda_1,\dots, \lambda_n$. We have 
$$
\int_{\bbR^n}\rho(\alpha^t {H}\alpha +u^tH^{-1}u)du_1\wedge\cdots\wedge du_n=\int_{\bbR^n}\rho(\alpha^tH \alpha +(Su)^t\Lambda Su)du_1\wedge\cdots\wedge du_n.
$$
Set $y=Su$ and change variables in the integral above to get
$$
\int_{\bbR^n}\rho(\alpha^tH \alpha +\sum_{i=1}^n\lambda_iy_i^2)dy_1\wedge\cdots\wedge dy_n.
$$
Now set $z_i=\sqrt \lambda_i y_i$ to get
\begin{equation}\label{eqn:especz}
\int_{\bbR^n}\rho(\alpha^t {H}\alpha +u^tH^{-1}u)du_1\wedge\cdots\wedge du_n=
\int_{\bbR^n}\frac{\rho(\alpha^tH \alpha +\sum_{i=1}^nz_i^2)}{\prod_{i=1}^n\sqrt \lambda_i}dz_1\wedge\cdots\wedge dz_n.
\end{equation}
Note that
$$
\frac{1}{\prod_{i=1}^n \sqrt \lambda_i}= \sqrt {\det H}.
$$
We will use spherical coordinates on $\bbR^n$ to write the right side of \eqref{eqn:especz} as
$$
C_n\sqrt{\det H}\int_{0}^{+\infty}\rho(\alpha^tH \alpha +r^2)r^{n-1}dr,
$$
where $C_n$ is a constant that depends only on $n$. It follows that 
$$
\int_P\sqrt{\det H}\int_{0}^{+\infty}\rho(\alpha^tH \alpha +r^2)r^{n-1}drdx
$$
is espectrally determined. 
By taking the limit of an appropriate sequence of functions that equal $\rho$ on larger and larger sets, we see that we can make {$\rho_\param(s)=e^{i\param s}$}; it follows that
\begin{equation}\label{eqn:espec1}
\int_P\sqrt{\det H}e^{i\param \alpha^tH \alpha} dx
\end{equation}
is espectrally determined because the integral
$$
\int_{0}^{+\infty}e^{i\param r^2}r^{n-1}dr
$$
is convergent.

Now fix $u \in \bbR^n$. Multiplying \eqref{eqn:espec1} by $\rho(\alpha)e^{{-}i\param \alpha^t u}$ for some $\rho \in C_0^\infty(\bbR)$ and integrating with respect to $\alpha$, it follows that
$$
\int_P\int_{\bbR^n}\sqrt{\det H}e^{i\param (\alpha^tH \alpha-\alpha^tu)}\rho(\alpha) d\alpha dx
$$
is espectrally determined. We are going to apply stationary phase to this integral when $\param = \frac{1}{2h}$ tends to infinity.  First, however, we rewrite the integral slightly. 
Set $b=-\frac{H^{-1}u}{2}$.
We have 
$$
(\alpha+b)^tH(\alpha+b)=\alpha^tH \alpha-\alpha^tu+\frac{u^tH^{-1}u}{4}
$$
so that 
$$
\int_{\bbR^n}e^{i\param (\alpha^tH \alpha-\alpha^tu)}\rho(\alpha) d\alpha=e^{-\frac{i\param u^tH^{-1}u}{4}}\int_{\bbR^n}e^{i\param ((\alpha+b)^tH (\alpha+b))}\rho(\alpha) d\alpha .
$$
Set $y=\alpha+b$ in the integral above. It follows that 
$$
\int_{\bbR^n}e^{i\param (\alpha^tH \alpha-\alpha^tu)}\rho(\alpha) d\alpha=e^{-\frac{i\param u^tH^{-1}u}{4}}\int_{\bbR^n}e^{i\param y^tH y}\rho\left(y+\frac{H^{-1}u}{2}\right) dy.
$$
We observe that we can apply stationary phase as in Lemma \ref{stationary_phase}  to get
$$
\int_{\bbR^n}e^{i\param y^tH y}\rho\left(y+\frac{H^{-1}u}{2}\right) dy=(2\pi h)^{\frac{n}{2}}|\det H|^{-1/2}e^{i\pi n/4}\left(\exp\left(\frac{-ihb(D)\tilde{\rho}}{{2}}\right)\right)(0),
$$ 
where $\tilde{\rho}(y)=\rho\left(y+\frac{H^{-1}u}{2}\right)$
and $b(D)$ is a constant coefficient differential operator of the form
$$
-\sum b_{kl}\frac{\partial }{\partial y_k}\frac{\partial }{\partial y_l},
$$
with $B=H^{-1}$.
It follows that 
$$
\int_{\bbR^n}e^{i\param (\alpha^tH \alpha-\alpha^tu)}\rho(\alpha) d\alpha=(2\pi h)^{\frac{n}{2}}e^{-\frac{i\param u^tH^{-1}u}{4}}\frac{e^{i\pi n/4}}{\sqrt {\det H}}\left(\exp\left(\frac{-ihb(D)\tilde{\rho}}{4}\right)\right)(0).
$$

We next determine the value of the exponential term involving $\tilde{\rho}$.  Since $H^{-1}$ is smooth on $P$, there exists $R$ such that $-\frac{H^{-1}u}{2} \in B_R(0)$ for all $x\in P$. Pick $\rho$ such that $\rho$ is identically $1$ on $B_{2R}$ and $0$ outside $B_{3R}$. All the derivatives of $\tilde{\rho}$ at $0$ are zero because the derivatives of $\rho$ at $\frac{H^{-1}u}{2}$ are zero. For such $\rho$ we have  
$$
\int_P\int_{\bbR^n}\sqrt{\det H}e^{i\param (\alpha^tH \alpha-\alpha^tu)}\rho(\alpha) d\alpha dx=(2\pi h)^{\frac{n}{2}}e^{i\pi n/4} \int_P e^{-\frac{i\param u^tH^{-1}u}{4}}dx
$$
and we may conclude that 
$$
\int_P e^{-\frac{i\param u^tH^{-1}u}{4}}dx
$$
is espectrally determined for all $u\in \bbR^n$. 

Consider the function
$$
\param \rightarrow \int_P e^{-\frac{i\param u^tH^{-1}u}{4}}dx.
$$
This function is a holomorphic function of $\param$ because $u^tH^{-1}u$ is smooth on $P$. Therefore its values on the positive real line determine it completely. Said in another way, if two holomorphic functions agree on the positive real line then they agree on $\bbC$. This means that for any positive $t$, the quantity
$$
\int_P e^{-tu^tH^{-1}u}dx
$$
is espectrally determined. But now consider the limit of this quantity as $t$ tends to $\infty$. The matrix $H^{-1}$ is strictly positive definite in the interior of $P$; near the boundary of $P$, the matrix $H^{-1}$ acquires a kernel spanned by the normals to the facets of $P$. When $u$ is distinct from all those normals, we have
$$
\lim_{t\rightarrow \infty} \int_P e^{-tu^tH^{-1}u}dx=0
$$
because $u^tH^{-1}u$ is uniformly bounded from below on $P$. If $u=\n_i$ for some $i\in \{1,\dots, 
{d}\}$ then we claim that
$$
\lim_{t\rightarrow \infty} \int_P e^{-tu^tH^{-1}u}dx=\sum_{j:F_j\perp \n_i} \int_{F_j} 1.
$$
{For some intuition about this claim, consider a general integral of the form $\int e^{-tf(x)} dx$ for some smooth function $f(x)$.  When we take the limit of this integral as $t$ tends to infinity, we see that the contributions to the integral will come at the points where $f(x)=0$.  In our situation, the third property of Lemma \ref{boundary_behavior} tells us that we can replace $\n_i^t H^{-1}\n_i$ by $l_i s(x)$, where $s(x)$ is a smooth function that does not vanish on $P$.  Thus contributions to the integral will come from the facets for which $\n_i$ is a normal vector. In particular, for each $\n_i$ normal to a facet of $P$ we recover the sums of the volumes of the facets perpendicular to $\n_i$. 

If $P$ has no parallel facets {and no subpolytopes}, our theorem follows from the existence of a positive solution to the Minkowski problem, namely that the facet normals and the facet volumes determine the polytope up to {two choices and up to} translation (e.g., \cite[Thm. 2]{k}).}
\end{proof}

The unlabeled polytope of a toric orbifold $(X, \omega)$ does not determine the toric orbifold up to symplectomorphism, but it does determine the fan associated to the toric variety $X$, and hence determines $X$ up to equivariant biholomorphism. See \cite[Thm. 9.4]{lt} for more details.


\section{Weighted projective spaces}\label{weighted}

Our applications of Theorem \ref{mainthm} have used only the $\mu_{0,0}$ term in the asymptotic expansion; in a subsequent paper, we plan to explore applications involving the {other} terms as well. In the toric setting we note that there are never points with discrete isotropy groups. Only points in the pre-image of the facets can have nontrivial isotropy, and those points always have isotropy groups of dimension at least 1.

 An example of an interesting situation where we could make use of these extra {terms} is as follows. Let $X$ be the unit sphere in $\bbC^{d}$ given by
 $$
 X=\{(z_1,\dots,z_d): \sum_{i=1}^d |z_i|^2=1\}.
 $$
 Let $S^1$ act on $X$ with weights $(m_1,\dots,m_d)$, i.e.,
 $$
 \action(e^{i\theta})(z_1,\dots,z_d)=(e^{im_1\theta}z_1,\dots,e^{im_d\theta}z_d).
 $$
{
\begin{prop}\label{prop:weights}
Let $X$ be endowed with an $S^1$-invariant metric. The asymptotic equivariant spectrum of the Laplace operator on $\Sm(X)^{\frac{1}{h}}$ determines the weights if they are all distinct primes.
\end{prop}}

This result follows from Theorem \ref{mainthm}.
Note that studying $S^1$-equivariant functions on $X$ is equivalent to studying invariant sections of a line bundle as in \S \ref{sec:quantum_reduc}, so our setting is in fact that of weighted projective spaces.

\begin{proof}
We start by determining the non-trivial isotropy groups of the $S^1$ action. Given $z\in X$ and $e^{i\theta}\in S^1$, $\action(e^{i\theta})z=z$ if and only if
$$
e^{i{m_k \theta}}z_k=z_k, \text{ \ for all } k=1,\dots, d,
$$
or equivalently,
$$
\theta=\frac{2\pi s}{m_k} \text{\ with\ } s\in \bbZ \text{\ or\ } z_k=0, \text{\ for all } k=1,\dots, d.
$$
Because we are assuming that $m_1,\dots, m_d$ are relatively prime, there cannot be two distinct values of $k$, say $a$ and $b$, for which $\theta=\frac{2\pi s_a}{m_{a}}$ and $\theta=\frac{2\pi s_b}{m_{b}}$ for integers $s_a$ and $s_b$. This implies that $z\in X$ has nontrivial isotropy if and only if $z$ has exactly one nonzero component. Set
$$
X_k=\{(z_1,\dots,z_d) \in X: z_j=0\,\text{\ for all } j\ne k\}.
$$
The $S^1$ action on $X$ is free on $X\setminus \bigcup_{k=1}^d X_k$.  If $z\in X_k$ then the isotropy group of $z$ is $\bbZ_{m_k}$, which is generated by the $m_k$th-roots of unity in $S^1$.  Also, the fixed point set of 
$e^{\frac{i2\pi s}{m_k}}$ equals $X_k,$ for all $s=1,\dots, m_k-1$. We conclude that the list of elements in $S^1$ with nontrivial fixed point sets are all the nontrivial $m_k$th-roots of unity for $k=1, \dots, d$, and the corresponding fixed sets are the one-dimensional manifolds $X_k$, $k=1, \dots, d$. The normal space to each $X_k$ in $X$ is 
$$
\text{Span}\Bigg\{\frac{\partial}{\partial x_j},\frac{\partial}{\partial y_j}, \, j\ne k\Bigg\},
$$
 where $x_j$ and $y_j$ are the real and imaginary parts of $z_j$, respectively. Note also that $e^{\frac{i2\pi s}{m_k}}$ acts by a nontrivial rotation on each subspace 
 $$
 \text{Span}\Bigg\{\frac{\partial}{\partial x_j},\frac{\partial}{\partial y_j}\Bigg\}
 $$ 
 for every $j\ne k$. Let $A_{s,k}$ denote the linearized $S^1$ action of $e^{\frac{i2\pi s}{m_k}}$ on the normal space to $X_k$, for $s = 1, \dots, m_k-1$.
 
Applying Theorem \ref{mainthm} in our setting with $\alpha=1$, we conclude that  
$$
\begin{aligned}
\mu_{\frac{1}{h}}(\rho) &\sim& (2\pi h)^{{2-2d}} \sum_{i=0}^{\infty} h^i \sum_{\ell \leq 2i} \mu_{i,\ell}{\left( \frac{d^{\ell}\rho}{dt^{\ell}} \right)} \hspace*{5cm} \notag \\ 
&&+\  {(2\pi h)^{-2}}\sum_{k=1}^d\sum_{s=1}^{m_k-1}  e^{-\frac{i2\pi s}{m_k h}}{|\operatorname{det}(A_{s,k}-I)|}^{-1} {(\p_1^{k,s})_* \nu_{k,s}} +O(h).
\end{aligned}
$$
The first line consists of terms that are real and the second line does not. We conclude that 
\begin{equation}\label{eqn:wpp}
\sum_{k=1}^d\sum_{s=1}^{m_k-1} e^{-\frac{i2\pi s}{m_k h}}{|\operatorname{det}(A_{s,k}-I)|}^{-1} {(\p_1^{k,s})_* \nu_{k,s}} 
\end{equation}
is espectrally determined. 

Now assume that $h=\frac{1}{M m_1 \cdots m_d }$ for some positive integer $M$. Then \eqref{eqn:wpp} becomes real. We can thus conclude that the product $m_1\cdots  m_d$ is espectrally determined, and therefore so is each $m_j$ as the $m_j$ are distinct primes. 
\end{proof}

Note that Proposition \ref{prop:weights} complements the results of \cite{adfg} and \cite{guw}.  
It is likely that by exploiting the terms arising from fixed points in the asymptotic expansion for $\mu_{\frac{1}{h}}$, the above proposition can be extended to more general values for the $m_k$'s. We will consider this in a subsequent paper.

\bibliographystyle{plain}
\bibliography{s-c_weights}

\end{document}